\documentclass[12pt]{amsart}
\usepackage[margin=1.2in]{geometry}
\usepackage{graphicx,latexsym}
\usepackage{comment}
\usepackage{tikz}
\usepackage{amsfonts, amssymb, amsmath, amsthm, bm, hyperref}

\newcommand{\N}{\mathbb{N}}
\newcommand{\Z}{\mathbb{Z}}
\newcommand{\R}{\mathbb{R}}
\newcommand{\C}{\mathbb{C}}

\newcommand{\dx}{{\rm d}x }
\newcommand{\du}{{\rm d}u }

\newcommand{\dss}{{\rm d}s }

\newcommand{\dt}{{\rm d}t }
\newcommand{\dxi}{{\rm d}\xi }

\newcommand{\dy}{{\rm d}y }

\newtheorem{theorem}{Theorem}[section]
\newtheorem{proposition}[theorem]{Proposition}

\newtheorem{lemma}[theorem]{Lemma}
\newtheorem{corollary}[theorem]{Corollary}

\theoremstyle{definition}
\newtheorem{definition}[theorem]{Definition}
\newtheorem{example} [theorem]{Example}

\theoremstyle{remark}
\newtheorem{remark}[theorem]{Remark}

\numberwithin{equation}{section}
\begin{document}
\title[Topological properties of convolutor spaces]{Topological properties of convolutor spaces via the short-time Fourier transform}
\author[A. Debrouwere]{Andreas Debrouwere}
\address{Department of Mathematics: Analysis, Logic and Discrete Mathematics, Ghent University, Krijgslaan 281, 9000 Gent, Belgium}
\email{Andreas.Debrouwere@UGent.be}
\thanks{A. Debrouwere was supported by  FWO-Vlaanderen, through the postdoctoral grant 12T0519N}

\author[J. Vindas]{Jasson Vindas}
\thanks{J. Vindas was supported by Ghent University, through the BOF-grants 01J11615 and 01J04017.}
\address{Department of Mathematics: Analysis, Logic and Discrete Mathematics, Ghent University, Krijgslaan 281, 9000 Gent, Belgium}
\email{Jasson.Vindas@UGent.be}

\begin{abstract}
We discuss the structural and topological properties of a general class of weighted $L^1$ convolutor spaces. Our theory simultaneously applies to weighted $\mathcal{D}'_{L^1}$ spaces as well as to convolutor spaces of the Gelfand-Shilov spaces $\mathcal{K}\{M_p\}$.  In particular, we characterize the sequences of weight functions $(M_p)_{p \in \N}$ for which the space of convolutors of $\mathcal{K}\{M_p\}$ is ultrabornological, thereby generalizing Grothendieck's classical result for the space $\mathcal{O}'_{C}$ of rapidly decreasing distributions. Our methods  lead to the first direct proof of the completeness of the space $\mathcal{O}_{C}$ of very slowly increasing smooth functions.
\end{abstract}

\maketitle

\section{Introduction}
In his fundamental book \cite{Schwartz}, Schwartz introduced the space $\mathcal{O}'_{C}(\mathbb{R}^{d})$ of rapidly decreasing distributions and showed that it is in fact the space of \emph{convolutors} of $\mathcal{S}(\mathbb{R}^{d})$; namely, a tempered distribution $f \in \mathcal{S}'(\R^d)$ belongs to $\mathcal{O}'_C(\R^d)$ if and only if $f \ast \varphi \in \mathcal{S}(\R^d)$ for all  $\varphi \in \mathcal{S}(\R^d)$ \cite[Thm.\ IX, p.\ 244]{Schwartz}. Moreover, for $f  \in \mathcal{O}'_C(\R^d)$ fixed, the mapping $\mathcal{S}(\R^d) \rightarrow \mathcal{S}(\R^d), \, \varphi \mapsto f \ast \varphi$ is continuous by the closed graph theorem. This characterization suggests to endow $\mathcal{O}'_C(\R^d)$ with the initial topology with respect to the mapping
 $$
\mathcal{O}'_C(\R^d) \rightarrow L_b(\mathcal{S}(\R^d), \mathcal{S}(\R^d)), \quad f \mapsto (\varphi \mapsto f \ast \varphi).
$$

This definition entails that $\mathcal{O}'_C(\R^d)$ is semi-reflexive and nuclear. A detailed study of the locally convex structure of $\mathcal{O}'_C(\R^d)$ was carried out by  Grothendieck in the last part of his doctoral thesis \cite{Grothendieck}.  He showed that this space is ultrabornological \cite[Chap.\ II, Thm.\ 16, p.\ 131]{Grothendieck} and that its strong dual is given by the space $\mathcal{O}_C(\R^d)$ of very slowly increasing smooth functions  \cite[Chap.\ II, p.\ 131]{Grothendieck}. Consequently,   $\mathcal{O}_C(\R^d)$ is complete and its strong dual is equal to $\mathcal{O}'_C(\R^d)$. See \cite{Bargetz, B-O, Larcher, L-W, Ortner} 
for modern works concerning these spaces. We would like to point out that ultrabornologicity is the crucial hypothesis for the application of many functional analytic tools such as De Wilde's open mapping and closed graph theorems \cite{M-V}, and the abstract Mittag-Leffler theorem \cite{Wengenroth}, which is very useful to solve surjectivity problems.

On the other hand, two classical results of Schwartz state that a distribution $f \in \mathcal{D}'(\R^d)$ belongs to the space $\mathcal{D}'_{L^1}(\R^d)$ of integrable distributions if and only if $f \ast \varphi \in L^1(\R^d)$ for all $\varphi \in \mathcal{D}(\R^d)$ \cite[Thm.\ XXV, p.\ 201]{Schwartz} and that the strong dual of  $\mathcal{D}'_{L^1}(\R^d)$ is given by $\mathcal{B}(\R^d)$  \cite[p.\ 203]{Schwartz}. Weighted $\mathcal{D}'_{L^1}(\R^d)$ spaces have been considered by Ortner and Wagner \cite{O-W} and, more recently, by Dimovski, Pilipovi\'c and the second author in the broader framework of distribution spaces associated to general translation-invariant Banach spaces \cite{D-P-V2015TIB}. The space $\mathcal{D}'_{L^1}(\R^d)$ and its weighted variants play an essential role in the convolution theory for distributions \cite{Schwartz-57,B-N-O,O-W-2,Wagner}. 
 
 The main goal of this article is to develop a unified approach towards these two types of results and, at the same time, considerably extend them. More precisely, let $\mathcal{W} =(w_N)_{N \in \N}$ be a (pointwise) increasing sequence of positive continuous functions. We define $L^1_{\mathcal{W}}(\R^d)$ as the Fr\'echet space consisting of all measurable functions $f$ on $\R^d$ such that $\|f\|_{L^1_{w_N}} := \|fw_N\|_{L^1} < \infty$ for all $N \in \N$ and 
$$
\mathcal{O}'_C(\mathcal{D},L^1_{\mathcal{W}}) := \{ f \in \mathcal{D}'(\R^d) \, : \, f \ast \varphi \in L^1_\mathcal{W}(\R^d) \mbox{ for all } \varphi \in \mathcal{D}(\R^d) \}.
$$
 As before, for $f \in \mathcal{O}'_C(\mathcal{D},L^1_{\mathcal{W}})$ fixed, the closed graph theorem implies that the mapping $\mathcal{D}(\R^d) \rightarrow L^1_{\mathcal{W}}(\R^d), \, \varphi \mapsto f \ast \varphi$ is continuous. We endow $\mathcal{O}'_C(\mathcal{D},L^1_{\mathcal{W}})$ with the initial topology with respect to the mapping
 $$
\mathcal{O}'_C(\mathcal{D},L^1_{\mathcal{W}}) \rightarrow L_b(\mathcal{D}(\R^d),L^1_{\mathcal{W}}(\R^d)), \quad f \mapsto (\varphi \mapsto f \ast \varphi) 
.$$
Next, for a positive continuous function $v$ on $\R^d$, we define $\mathcal{B}_{v}(\R^d)$ as the Fr\'echet space consisting of all $\varphi \in C^\infty(\R^d)$ such that $\| \varphi \|_{v, n} := \max_{|\alpha| \leq n} \| (\partial^\alpha \varphi) v\|_{L^\infty} < \infty$ for all $n \in \N$ and  $\dot{\mathcal{B}}_{v}(\R^d)$ as the closure of $\mathcal{D}(\R^d)$ in $\mathcal{B}_{v}(\R^d)$. Finally, we introduce the $(LF)$-spaces
$$
\mathcal{B}_{\mathcal{W}^\circ}(\R^d) := \varinjlim_{N \in \N} \mathcal{B}_{1/w_N}(\R^d), \qquad \dot{\mathcal{B}}_{\mathcal{W}^\circ}(\R^d) := \varinjlim_{N \in \N} \dot{\mathcal{B}}_{1/w_N}(\R^d).
$$
Since $\mathcal{D}(\R^d)$ is dense in  $\dot{\mathcal{B}}_{\mathcal{W}^\circ}(\R^d)$,  we may view its dual $(\dot{\mathcal{B}}_{\mathcal{W}^\circ}(\R^d))'$ as a space of distributions. The main results of the present article can then be stated as follows:
\begin{theorem}\label{thm-introduction}
Let $\mathcal{W} =(w_N)_{N \in \N}$ be an increasing sequence of positive continuous functions and suppose that
\begin{equation}
\forall N \in \N\, \exists M \geq N \, : \,  \sup_{x \in \R^d}\frac{w_{N}(x + \: \cdot \:)}{w_{M}(x)} \in L^\infty_{\operatorname{loc}}(\R^d).
\label{general-cond}
\end{equation}
Then, $(\dot{\mathcal{B}}_{\mathcal{W}^\circ}(\R^d))' = \mathcal{O}'_C(\mathcal{D},L^1_{\mathcal{W}})$ as sets. Moreover, the following statements are equivalent:
\begin{itemize}
\item[$(i)$] $\mathcal{W}$ satisfies the condition:
\begin{equation}
\forall N \in \N \, \exists M \geq N \, \forall K \geq M \, \exists \theta \in (0,1) \, \exists C > 0 \, \forall x \in \R^d: 
 \label{Omega-switched}
\end{equation}
$$
 {w_N(x)}^{1-\theta}{w_K(x)}^{\theta} \leq Cw_M(x). 
$$
\item[$(ii)$] $\mathcal{B}_{\mathcal{W}^\circ}(\R^d)$ is complete.
\item[$(iii)$] $\dot{\mathcal{B}}_{\mathcal{W}^\circ}(\R^d)$ is complete.
\item[$(iv)$] $\mathcal{O}'_C(\mathcal{D},L^1_{\mathcal{W}})$ is ultrabornological.
\item[$(v)$]  $(\dot{\mathcal{B}}_{\mathcal{W}^\circ}(\R^d))'_b = \mathcal{O}'_C(\mathcal{D},L^1_{\mathcal{W}})$.
\end{itemize}
In such a case, the strong bidual of $\dot{\mathcal{B}}_{\mathcal{W}^\circ}(\R^d)$ is given by $\mathcal{B}_{\mathcal{W}^\circ}(\R^d)$. 
\end{theorem}

Condition \eqref{Omega-switched} means that $L^1_{\mathcal{W}}(\R^d)$ has property $(\Omega)$ of Vogt and Wagner \cite{M-V}. Hence, the equivalence between $(i)$ and $(iv)$ in Theorem \ref{thm-introduction} may be anticipated from the fact that, for a Fr\'echet space $E$, the operator space $L_b(\mathcal{D}(\R^d), E) \cong L_b(s, E)^\N$ is ultrabornological if $E$ has $(\Omega)$, while the converse implication holds if, e.g., $E$ is a K\"{o}the sequence space or $E$ is nuclear (cf. \cite[Thm.\ 4.1 and  4.9]{Vogt-2}).  Condition $(\Omega)$ plays an important role in the splitting theory for Fr\'echet spaces \cite{Vogt-2,Vogt} and, in fact, the ideas of some of our proofs in Section \ref{sect-indlimit-smooth} stem from this theory.

By applying Theorem \ref{thm-introduction} to a constant sequence $\mathcal{W} = (w)_{N \in \N}$, where $w$ is a positive continuous\footnote{In fact, as we shall show in Theorem \ref{weighted-dual-DL1}, it suffices to assume that $w$ is measurable.} function on $\R^d$ such that 
$\sup_{x \in \R^d}w(x + \: \cdot \:)/w(x) \in L^\infty_{\operatorname{loc}}(\R^d),$ we obtain the analogue of Schwartz' results for a very general class of weighted $\mathcal{D}'_{L^1}(\R^d)$ spaces (cf.\ Theorem \ref{weighted-dual-DL1}); see \cite{D-P-V2015TIB} for earlier work in this direction. Actually, to the best of our knowledge, the full topological identity $
\mathcal{D}'_{L^1}(\R^d) = \mathcal{O}'_C(\mathcal{D}, L^1)$ even seems to be new in the unweighted case; Schwartz only showed that these spaces coincide algebraically and have the same bounded sets and null sequences \cite[p.\ 202]{Schwartz}. 

We shall also show that, under natural assumptions on the sequence $\mathcal{W} =(w_N)_{N \in \N}$, the space $\mathcal{O}'_C(\mathcal{D},L^1_{\mathcal{W}})$ coincides topologically with the space of convolutors of the Gelfand-Shilov type spaces $\mathcal{B}_{\mathcal{W}}(\R^d) := \varprojlim_{N \in \N} \mathcal{B}_{w_N}(\R^d).$
Hence, Theorem \ref{thm-introduction} also comprises a quantified version of Grothendieck's results (cf.\ Theorem \ref{char-UB-smooth-GS}). The space of convolutors of the space $\mathcal{K}_1(\R^d)$ of exponentially decreasing smooth functions \cite{Hasumi} was studied by Ziele\'{z}ny \cite{Zielezny}. He claims that this space is ultrabornological but his argument seems to contain a gap (see Remark \ref{exp-example}); Theorem \ref{char-UB-smooth-GS} contains this result as a particular instance. 

It should be pointed out that the methods to be employed are completely different from the ones used by Schwartz and Grothendieck. Namely, we first introduce weighted $(LF)$-spaces of smooth functions and study their completeness; these spaces are the analogue of $\mathcal{O}_C(\R^d)$ in the present setting.  To this end, we shall use abstract results concerning the regularity properties of $(LF)$-spaces \cite{Wengenroth-96, Wengenroth}, which have their roots in Palamadov's homological theory for $(LF)$-spaces \cite{Palamadov}. Interestingly, when specialized to $\mathcal{O}_{C}(\mathbb{R}^{d})$, our method supplies what appears to be the first known direct proof in the literature of the completeness of $\mathcal{O}_{C}(\mathbb{R}^{d})$; we refer to \cite{L-W} for a direct proof of the dual statement, that is, the ultrabornologicity of $\mathcal{O}'_C(\R^d) \cong \mathcal{O}_M(\R^d)$. In the second part of this work, we shall exploit the mapping properties of the \emph{short-time Fourier transform} $(STFT)$ \cite{Grochenig} on various function and distribution spaces to show that $(\dot{\mathcal{B}}_{\mathcal{W}^\circ}(\R^d))' = \mathcal{O}'_C(\mathcal{D},L^1_{\mathcal{W}})$ as sets and to link the topological properties of $\mathcal{O}'_C(\mathcal{D},L^1_{\mathcal{W}})$ to those of $\mathcal{B}_{\mathcal{W}^\circ}(\R^d)$ and $\dot{\mathcal{B}}_{\mathcal{W}^\circ}(\R^d)$.  The proof of Theorem \ref{thm-introduction} will then be achieved by combining this with the results obtained in the first part of the article. In our opinion, the use of the STFT leads to transparent and insightful proofs of rather subtle results.
In this context, we highlight the papers \cite{B-O} in which the mapping properties of the STFT on  $\mathcal{O}'_C(\R^d)$ are established by using Schwartz' theory of vector-valued distributions and \cite{K-P-S-V} in which the STFT is used to characterize weighted $\mathcal{B}'(\R^d)$ and $\dot{\mathcal{B}}'(\R^d)$ spaces  in terms of the growth of convolution averages of their elements. We were inspired by both of these works. Finally, we would also like to mention that ultradistributional analogues of our results are treated in the paper \cite{D-Vconvolutorsultra}.

The plan of the article is as follows. In the auxiliary Section \ref{STFT-distributions}, we develop a theory of the STFT that applies to \emph{all} elements of $\mathcal{D}'(\R^d)$. This framework will enable us to deal with weight systems satisfying the very general condition \eqref{general-cond}. The locally convex structure of the $(LF)$-spaces $\mathcal{B}_{\mathcal{W}^\circ}(\R^d)$ and $\dot{\mathcal{B}}_{\mathcal{W}^\circ}(\R^d)$ is analyzed in Section \ref{sect-indlimit-smooth}. We shall show that for these spaces all regularity conditions considered in the literature (in particular, completeness) are equivalent to the fact that the weight system $\mathcal{W}$ satisfies condition \eqref{Omega-switched}. For later use, we also characterize these spaces in terms of the STFT. In the main Section \ref{L1-convolutors}, we study various structural and topological properties of the space $\mathcal{O}'_C(\mathcal{D},L^1_{\mathcal{W}})$ via the STFT and complete the proof of Theorem \ref{thm-introduction}. Our results concerning weighted $\mathcal{D}'_{L^1}$ spaces are presented in Section \ref{sect-weighted-DL1}. Finally, in Section \ref{GS-convolutors}, we apply our general theory to discuss the spaces of convolutors of $\dot{\mathcal{B}}_{\mathcal{W}}(\R^d)$.

\section{The short-time Fourier transform for general distributions}\label{STFT-distributions}
The aim of  this auxiliary section is to define and study the STFT for general distributions with respect to compactly supported smooth window functions. Most notably, we prove  reconstruction and desingularization formulas. These formulas will play an important role in the rest of this article. 

Our notation from distribution theory is standard. Given a compact $K \Subset \R^d$ and $n \in \N$, we write $\mathcal{D}_K^n$ for the Banach space consisting of all $\varphi \in C^n(\R^d)$ with $\operatorname{supp} \varphi \subseteq K$ endowed with the norm
$$
\| \varphi\|_{K,n} := \max_{|\alpha| \leq n} \max_{x \in K } |\partial^\alpha \varphi(x)|.
$$
For $n = 0$, we simply write $\|\: \cdot \: \|_{K} = \|\:\cdot \:\|_{K,0}$. We define
$$
\mathcal{D}_K := \varprojlim_{n \in \N} \mathcal{D}^n_K, \qquad \mathcal{D}(\R^d) := \varinjlim_{K \Subset \R^d} \mathcal{D}_K.
$$
Furthermore, we denote by $\mathcal{E}(\R^d)$ and $\mathcal{S}(\R^d)$ the space of smooth functions on $\R^d$ and the space of rapidly decreasing smooth functions on $\R^d$, respectively, each endowed with their standard Fr\'echet space structure. The dual spaces $\mathcal{D}'(\R^d)$, $\mathcal{E}'(\R^d)$ and $\mathcal{S}'(\R^d)$ are the space of distributions on $\R^d$, the space of compactly supported distributions on $\R^d$, and the space of tempered distributions on $\R^d$, respectively. We endow these spaces with their strong topologies.

Next, we recall some fundamental properties of the STFT on the space $L^2(\R^d)$; for further properties of the STFT we refer to the book \cite{Grochenig}. As customary, the translation and modulation operators are denoted by $T_xf = f(\:\cdot\: - x)$ and $M_\xi f = e^{2\pi i \xi \cdot} f$, $x, \xi \in \R^d$, respectively. We also write $\check{f} = f(- \:\cdot\:)$ for reflection about the origin. The STFT of a function $f \in L^2(\R^d)$ with respect to a window function $\psi \in L^2(\R^d)$ is defined as
$$
V_\psi f(x,\xi) := (f, M_\xi T_x\psi)_{L^2} = \int_{\R^d} f(t) \overline{\psi(t-x)}e^{-2\pi i \xi t} \dt, \qquad (x, \xi) \in \R^{2d}.
$$
We have that $\|V_\psi f\|_{L^2(\R^{2d})} = \|\psi\|_{L^2}\|f\|_{L^2}$. In particular, the mapping $V_\psi : L^2(\R^d) \rightarrow L^2(\R^{2d})$ is continuous. The adjoint of $V_\psi$ is given by the weak integral
$$
V^\ast_\psi F = \int \int_{\R^{2d}} F(x,\xi) M_\xi T_x\psi \dx \dxi, \qquad F \in L^2(\R^{2d}).
$$
If $\psi \neq 0$ and $\gamma \in L^2(\R^d)$ is a synthesis window for $\psi$, that is, $(\gamma, \psi)_{L^2} \neq 0$, then
\begin{equation}
\frac{1}{(\gamma, \psi)_{L^2}} V^\ast_\gamma \circ V_\psi = \operatorname{id}_{L^2(\R^d)}.
\label{reconstruction-L2}
\end{equation}

In order to be able to extend the STFT to the space of distributions $\mathcal{D}'(\R^d)$ we must first establish the mapping properties of the STFT on $\mathcal{D}(\R^d)$. We need some preparation. Given two lcHs (= locally convex Hausdorff spaces) $E$ and $F$, we write $E \widehat{\otimes}_\pi F$, $E \widehat{\otimes}_\varepsilon F$, and $E \widehat{\otimes}_i F$ for the completion of $E \otimes F$ with respect to the projective topology, the $\varepsilon$-topology and the inductive topology, respectively \cite{Komatsu3}. If either $E$ or $F$ is nuclear, we simply write $E \widehat{\otimes} F  = E \widehat{\otimes}_\pi F = E \widehat{\otimes}_\varepsilon F$. Let $d_1, d_2 \in \N$ and $K \Subset \R^{d_1}$. We may identify the space $\mathcal{D}_{K,x} \widehat{\otimes} \mathcal{S}(\R^{d_2}_\xi)$ with the Fr\'echet space consisting of all $\varphi \in C^\infty(\R^{d_1 + d_2}_{x,\xi})$ with $\operatorname{supp} \varphi \subseteq K \times \R^{d_2}$ such that
\begin{equation}
|\varphi|_{K,n} := \max_{|\alpha| \leq n} \max_{|\beta| \leq n} \sup_{(x,\xi) \in K \times \R^{d_2}} |\partial^{\alpha}_x \partial^\beta_\xi \varphi(x,\xi)|(1+|\xi|)^n < \infty 
\label{definition-norm-tensor}
\end{equation}
for all $n \in \N$. By \cite[Thm.~2.3, p.~670]{Komatsu3}, we have the following canonical isomorphisms of lcHs
$$
\mathcal{D}(\R^{d_1}_x) \widehat{\otimes}_i \mathcal{S}(\R_\xi^{d_2}) \cong \varinjlim_{K \Subset \R^{d_1}} \mathcal{D}_{K,x} \widehat{\otimes} \mathcal{S}(\R^{d_2}_\xi)
$$
and
$$
(\mathcal{D}(\R^{d_1}_x) \widehat{\otimes}_i \mathcal{S}(\R_\xi^{d_2}))'_{b} \cong  \mathcal{D}'(\R^{d_1}_x) \widehat{\otimes} \mathcal{S}'(\R^{d_2}_\xi).
$$
We then have:
\begin{proposition}\label{STFT-D}
Let $\psi \in \mathcal{D}(\R^d)$. Then, the mappings
$$
V_\psi: \mathcal{D}(\R^d) \rightarrow \mathcal{D}(\R^d_x) \widehat{\otimes}_i \mathcal{S}(\R_\xi^d)
$$
and
$$
V^\ast_\psi:  \mathcal{D}(\R^d_x) \widehat{\otimes}_i \mathcal{S}(\R_\xi^d) \rightarrow \mathcal{D}(\R^d)
$$
are well-defined and continuous.
\end{proposition}
\begin{proof}
We first look at $V_\psi$. Consider the continuous linear mappings
$$
S: \mathcal{D}(\R^d_t) \rightarrow \mathcal{D}(\R^{2d}_{x,t}), \quad \varphi(t) \mapsto \varphi(t)T_x{\overline{\psi}}(t),
$$
and
$$
\mathcal{F}_t: \mathcal{D}(\R^{2d}_{x,t}) \rightarrow \mathcal{D}(\R^d_x) \widehat{\otimes}_i \mathcal{S}(\R_\xi^d), \quad \varphi(x,t) \mapsto \int_{\R^d} \varphi(x,t)e^{-2\pi i \xi t} \dt.
$$
The result follows from the representation $V_\psi = \mathcal{F}_t \circ S$. Next, we treat $V_\psi^\ast$. Consider the continuous linear mappings
$$
S: \mathcal{D}(\R^d_x) \widehat{\otimes}_i \mathcal{S}(\R_\xi^d) \rightarrow \mathcal{D}(\R^d_t) \widehat{\otimes}_i \mathcal{S}(\R_{x,\xi}^{2d}), \quad \varphi(x,\xi) \mapsto \varphi(x,\xi)M_\xi T_x\psi(t)
$$
and
$$
\operatorname{id}_{ \mathcal{D}(\R^d_t)} \widehat{\otimes}_i 1(x,\xi) :  \mathcal{D}(\R^d_t) \widehat{\otimes}_i \mathcal{S}(\R_{x,\xi}^{2d}) \rightarrow \mathcal{D}(\R^d_t), \quad \varphi(t,x,\xi) 
\mapsto \int\int_{\R^{2d}} \varphi(t,x,\xi) \dx\dxi.
$$
The result follows from the representation $V^\ast_\psi = (\operatorname{id}_{ \mathcal{D}(\R^d_t)}  \widehat{\otimes}_i 1(x,\xi) )\circ S$.
\end{proof}
Observe that, if $\psi \in \mathcal{D}(\R^d) \backslash \{0\}$ and $\gamma \in \mathcal{D}(\R^d)$ is a synthesis window for $\psi$,
the reconstruction formula \eqref{reconstruction-L2} implies that
\begin{equation}
\frac{1}{(\gamma, \psi)_{L^2}} V^\ast_\gamma \circ V_\psi = \operatorname{id}_{\mathcal{D}(\R^d)}.
\label{reconstruction-D}
\end{equation}

We are ready to define the STFT on the space $\mathcal{D}'(\R^d)$. For $\psi \in \mathcal{D}(\R^d)$ and $f \in \mathcal{D}'(\R^d)$ we define
$$
V_\psi f(x,\xi) :=  \langle f, \overline{M_\xi T_x\psi} \rangle = e^{-2\pi i \xi x} (f \ast M_\xi \check{\overline{\psi}})(x), \qquad (x,\xi) \in \R^{2d}.
$$
Clearly, $V_\psi f$ is a smooth function on $\R^{2d}$.
\begin{lemma} \label{well-defined-STFT}
Let $\psi \in \mathcal{D}(\R^d)$ and  $f \in  \mathcal{D}'(\R^d)$. Then, for every $K \Subset \R^d$ there is $n \in \N$ such that
$$
\sup_{(x,\xi) \in K \times \R^{d}}\frac{|V_\psi f(x,\xi)|}{(1+|\xi|)^n} < \infty.
$$
In particular, $V_\psi f$ defines an element of $\mathcal{D}'(\R^{d}_x) \widehat{\otimes} \mathcal{S}'(\R^{d}_\xi)$ via
$$
\langle V_\psi f, \varphi \rangle := \int \int_{\R^{2d}} V_\psi f(x,\xi) \varphi(x,\xi) \dx \dxi, \qquad \varphi \in \mathcal{D}(\R^d_x) \widehat{\otimes}_i \mathcal{S}(\R_\xi^d).
$$
\end{lemma}
\begin{proof}
Let $K \Subset \R^d$ be arbitrary and set $L = \operatorname{supp} \psi + K$. Since $\operatorname{supp} \overline{M_\xi T_x\psi} \subseteq L$ for all $(x,\xi) \in K \times \R^d$, there are $n \in \N$ and $C > 0$ such that
$$
|V_\psi f(x,\xi)| = | \langle f, \overline{M_\xi T_x\psi} \rangle| \leq C \|  M_\xi T_x\psi \|_{L,n}, \qquad (x,\xi) \in K \times \R^d.
$$
The result follows from the fact that
\begin{align*}
\| M_\xi T_x\psi\|_{L,n} 
&\leq \max_{|\alpha| \leq n} \sum_{\beta \leq \alpha} \binom{\alpha}{\beta}(2\pi|\xi|)^{|\beta|} \sup_{t \in L} | \partial^{\alpha -\beta} \psi(t-x)| \\
& \leq (2\pi)^n \|\psi\|_{\operatorname{supp} \psi,n} (1+ |\xi|)^n 
\end{align*}
for all $(x,\xi) \in K \times \R^d$.
\end{proof}

\begin{lemma}\label{STFT-transpose}
Let $\psi \in \mathcal{D}(\R^d)$ and let $f \in  \mathcal{D}'(\R^d)$. Then, 
$$
\langle V_\psi f, \varphi \rangle = \langle f, \overline{V^\ast_\psi \overline{\varphi}} \rangle, \qquad \varphi \in \mathcal{D}(\R^d_x) \widehat{\otimes}_i \mathcal{S}(\R_\xi^d).  
$$
\end{lemma}
\begin{proof}
Let $\varphi \in \mathcal{D}(\R^d_x) \widehat{\otimes}_i \mathcal{S}(\R_\xi^d)$ be arbitrary.  Since  
$$
\varphi(x,\xi)\overline{M_\xi T_x \psi}(t) \in \mathcal{D}(\R^d_t)  \widehat{\otimes}_i \mathcal{S}(\R_{x,\xi}^{2d})
$$
and 
$$
1(x,\xi) \widehat{\otimes}_i f(t)  = f(t) \widehat{\otimes}_i 1(x,\xi) \in \mathcal{D}'(\R^{d}_t) \widehat{\otimes} \mathcal{S}'(\R^{2d}_{x,\xi}),
$$
we have that
\begin{align*}
\langle V_\psi f, \varphi \rangle & = \int \int_{\R^{2d}} \langle f, \overline{M_\xi T_x\psi} \rangle \varphi(x,\xi) \dx \dxi = \langle 1(x,\xi) \widehat{\otimes}_i f(t),  \varphi(x,\xi)\overline{M_\xi T_x\psi}(t) \rangle  \\
&= \langle f(t) \widehat{\otimes}_i 1(x,\xi) ,  \varphi(x,\xi)\overline{M_\xi T_x\psi}(t) \rangle = \langle f(t), \int \int_{\R^{2d}} \varphi(x,\xi)\overline{M_\xi T_x\psi}(t) \dx \dxi \rangle \\
&=\langle f, \overline{V^\ast_\psi \overline{\varphi}} \rangle.
\qedhere
\end{align*}
\end{proof}
Let $\psi \in \mathcal{D}(\R^d)$. We \emph{define} the adjoint STFT of an element $F \in \mathcal{D}'(\R^{d}_x) \widehat{\otimes} \mathcal{S}'(\R^{d}_\xi)$ as
$$
\langle V^\ast_\psi F, \varphi \rangle := \langle F, \overline{V_\psi \overline{\varphi}} \rangle, \qquad \varphi \in \mathcal{D}(\R^d). 
$$
Notice that $V^\ast_\psi F \in \mathcal{D}'(\R^d)$ because of Proposition \ref{STFT-D}. We  have all the necessary ingredients to establish the mapping properties of the STFT on $\mathcal{D}'(\R^d)$.
\begin{proposition}\label{STFT-D-dual}
Let $\psi \in \mathcal{D}(\R^d)$. Then, the mappings
$$
V_\psi: \mathcal{D}'(\R^d) \rightarrow \mathcal{D}'(\R^{d}_x) \widehat{\otimes} \mathcal{S}'(\R^{d}_\xi)
$$
and
$$
V^\ast_\psi: \mathcal{D}'(\R^{d}_x) \widehat{\otimes} \mathcal{S}'(\R^{d}_\xi) \rightarrow \mathcal{D}'(\R^d)
$$
are well-defined and continuous. Moreover, if $\psi \neq 0$ and $\gamma \in \mathcal{D}(\R^d)$ is a synthesis window for $\psi$, then
\begin{equation}
\frac{1}{(\gamma, \psi)_{L^2}} V^\ast_\gamma \circ V_\psi = \operatorname{id}_{\mathcal{D}'(\R^d)}
\label{reconstruction-D-dual}
\end{equation}
and the desingularization formula 
\begin{equation}
\langle f, \varphi \rangle = \frac{1}{(\gamma, \psi)_{L^2}} \int \int_{\R^{2d}} V_\psi f(x,\xi) V_{\overline{\gamma}} \varphi(x,-\xi) \dx \dxi
\label{desing-D-dual}
\end{equation}
holds for all $f \in \mathcal{D}'(\R^d)$ and $\varphi \in \mathcal{D}(\R^d)$. 
\end{proposition}
\begin{proof}
The mapping  $V_\psi$ is continuous because of Lemma \ref{STFT-transpose} and the continuity of $V^\ast_\psi: \mathcal{D}(\R^d_x) \widehat{\otimes}_i \mathcal{S}(\R_\xi^d) \rightarrow \mathcal{D}(\R^d)$ (Proposition \ref{STFT-D}), while $V^\ast_\psi$ is continuous because  of the continuity of $V_\psi: \mathcal{D}(\R^d) \rightarrow \mathcal{D}(\R^d_x) \widehat{\otimes}_i \mathcal{S}(\R_\xi^d)$ (Proposition \ref{STFT-D}). Next, suppose that $\psi \neq 0$ and $\gamma \in \mathcal{D}(\R^d)$ is a synthesis window for $\psi$. Let $f \in \mathcal{D}'(\R^d)$ and $\varphi \in \mathcal{D}(\R^d)$ be arbitrary. Lemma \ref{STFT-transpose} and the reconstruction formula \eqref{reconstruction-D} imply that
$$
\langle V^\ast_\gamma (V_\psi f), \varphi \rangle = \langle V_\psi f, \overline{V_\gamma \overline{\varphi}} \rangle = \langle f,  \overline{V^\ast_\psi (V_\gamma \overline{\varphi})} \rangle = (\gamma, \varphi)_{L^2} \langle f, \varphi \rangle;
$$
that is, \eqref{reconstruction-D-dual} and \eqref{desing-D-dual} hold.
\end{proof}

Finally, we show that, for $f \in \mathcal{E}'(\R^d)$, the desingularization formula \eqref{desing-D-dual} holds for all $\varphi \in \mathcal{E}(\R^d)$. To this end, we briefly discuss the STFT on $\mathcal{E}(\R^d)$ and $\mathcal{E}'(\R^d)$. The space $\mathcal{E}(\R^d_x) \widehat{\otimes} \mathcal{S}(\R_\xi^d)$ may be identified with the Fr\'echet space consisting of all $\varphi \in C^\infty(\R^{2d}_{x,\xi})$ such that $|\varphi|_{K,n} < \infty$ for all $K \Subset \R^d$ and $n \in \N$ (cf.\ \eqref{definition-norm-tensor}).
\begin{proposition}\label{STFT-E}
Let $\psi \in \mathcal{D}(\R^d)$. 
\begin{itemize}
\item[$(i)$] The mapping
$$
V_\psi: \mathcal{E}(\R^d) \rightarrow \mathcal{E}(\R^d_x) \widehat{\otimes} \mathcal{S}(\R_\xi^d)
$$
is well-defined and continuous.
\item[$(ii)$] Let $f \in \mathcal{E}'(\R^d)$. Then, there are a compact $K \Subset \R^d$ and $n \in \N$ such that $\operatorname{supp} V_\psi f \subseteq K \times \R^d$ and 
$$
\sup_{(x,\xi) \in K \times \R^{d}}\frac{|V_\psi f(x,\xi)|}{(1+|\xi|)^n} < \infty.
$$
In particular, $V_\psi f$ defines an element of $\mathcal{E}'(\R^d_x) \widehat{\otimes} \mathcal{S}'(\R_\xi^d) \cong (\mathcal{E}(\R^d_x) \widehat{\otimes} \mathcal{S}(\R_\xi^d))'$ via
$$
\langle V_\psi f, \varphi \rangle := \int \int_{\R^{2d}} V_\psi f(x,\xi) \varphi(x,\xi) \dx \dxi, \qquad \varphi \in \mathcal{E}(\R^d_x) \widehat{\otimes} \mathcal{S}(\R_\xi^d).
$$
\end{itemize}
\end{proposition}
\begin{proof}
$(i)$ It suffices to observe that we can factor $V_\psi = \mathcal{F}_t \circ S$ through the continuous linear mappings
$$
S: \mathcal{E}(\R^d_t) \rightarrow \mathcal{E}(\R^d_x) \widehat{\otimes} \mathcal{S}(\R_t^d), \quad \varphi(t) \mapsto \varphi(t)T_x{\overline{\psi}}(t)
$$
and
$$
\mathcal{F}_t: \mathcal{E}(\R^d_x) \widehat{\otimes} \mathcal{S}(\R_t^d) \rightarrow \mathcal{E}(\R^d_x) \widehat{\otimes}\mathcal{S}(\R_\xi^d), \, \varphi(x,t) \mapsto \int_{\R^d} \varphi(x,t)e^{-2\pi i \xi t} \dt.
$$

$(ii)$ Since $\operatorname{supp} \overline{M_\xi T_x \psi} \subseteq \operatorname{supp} \psi + x$ for all $(x,\xi) \in \R^{2d}$, we obtain that $V_\psi f(x,\xi) = \langle f, \overline{M_\xi T_x \psi} \rangle = 0$ for all $(x, \xi) \notin (\operatorname{supp} f - \operatorname{supp} \psi) \times \R^d$, that is, $\operatorname{supp} V_\psi f \subseteq (\operatorname{supp} f - \operatorname{supp} \psi) \times \R^d$. The second part follows from Lemma \ref{well-defined-STFT}.
\end{proof}
\begin{corollary}\label{desing-E-dual}
Let $\psi \in \mathcal{D}(\R^d) \backslash \{0\}$ and $\gamma \in \mathcal{D}(\R^d)$ be a synthesis window for $\psi$. Then, the desingularization formula 
\eqref{desing-D-dual} holds for all $f \in \mathcal{E}'(\R^{d})$ and $\varphi \in \mathcal{E}(\R^d)$. 
\end{corollary}
\begin{proof}
Let $f \in \mathcal{E}'(\R^d)$ and  $\varphi \in \mathcal{E}(\R^d)$ be arbitrary. Choose a sequence $(\varphi_n)_{n \in \N} \subset \mathcal{D}(\R^d)$ such that $\varphi_n \rightarrow \varphi$ in $\mathcal{E}(\R^d)$. Hence, the desingularization formula \eqref{desing-D-dual} and Proposition \ref{STFT-E} imply that
$$
\langle f, \varphi \rangle = \lim_{n \to \infty} \langle f, \varphi_n \rangle = \lim_{n \to \infty} \langle V_\psi f, \overline{V_\gamma \overline{\varphi_n}} \rangle =  \langle V_\psi f, \overline{V_\gamma \overline{\varphi}} \rangle,
$$
so that \eqref{desing-D-dual} holds for $f$ and $\varphi$.
\end{proof}
\section{Weighted inductive limits of spaces of smooth functions}\label{sect-indlimit-smooth}
In this section, we introduce two general classes of weighted inductive limits of spaces of smooth functions defined via a decreasing sequence of positive continuous functions. Our main goal is to characterize the completeness of these spaces in terms of the defining sequence of weight functions. In order to do so, we first recall several regularity conditions for $(LF)$-spaces. Furthermore, we also establish the mapping properties of the STFT on these spaces; we shall repeatedly use the latter results in Sections \ref{L1-convolutors} and \ref{GS-convolutors}.
\subsection{Regularity conditions for $(LF)$-spaces}\label{sect-reg}
A lcHs $E$ is called an $(LF)$-space if there is a sequence $(E_N)_{N \in \N}$ of Fr\'echet spaces with $E_N \subset E_{N + 1}$ and continuous inclusion mappings such that $E = \bigcup_{N \in \N} E_N$ and the topology of $E$ coincides with the finest locally convex topology for which all inclusion mappings $E_n \rightarrow E$ are continuous. We emphasize that, for us, $(LF)$-spaces are Hausdorff by definition. We call  $(E_N)_{N}$ a defining inductive spectrum for $E$ and write $E = \varinjlim_{N}E_N$. If each $E_N$ is a Banach space, $E$ is called an $(LB)$-space. 

Let $E = \varinjlim_{N}E_N$ be an $(LF)$-space. We shall consider the following regularity conditions on $E$:
\begin{itemize}
\item[$(i)$] $E$ is said to be \emph{boundedly retractive} if for every bounded set $B$ in $E$ there is $N \in \N$ such that $B$ is contained in $E_N$, and $E$ and $E_N$ induce the same topology on $B$.
\item[$(ii)$] $E$ is said to be \emph{regular} if for every bounded set $B$ in $E$ there is $N \in \N$ such that $B$ is contained and bounded in $E_N$.
\item[$(iii)$] $E$ is said to be $\beta$-\emph{regular} if for every $N \in \N$ and every subset $B$ of $E_N$ that is bounded in $E$ there is $M \geq N$ such that $B$ is bounded in $E_M$.
\item[$(iv)$] $E$ is said to satisfy condition $(wQ)$ if for every $N \in \N$ there are a neighborhood $U$ of $0$ in $E_N$ and $M \geq N$ such that for every $K \geq M$ and every neighborhood $W$ of $0$ in $E_M$ there are a neighborhood $V$ of $0$ in $E_K$ and $\lambda > 0$ such that $V \cap U \subseteq \lambda W$. If $ (\| \: \cdot \: \|_{N,n})_{n \in \N}$ is a fundamental increasing sequence of seminorms for $E_N$, then $E$ satisfies $(wQ)$ if and only if
\begin{gather*}
\forall N \in \N \, \exists M \geq N \, \exists n \in \N \, \forall K \geq M \, \forall m \in \N \, \exists k \in \N \, \exists C > 0\, \forall e \in E_N: \\
\|e\|_{M,m} \leq C(\|e\|_{N,n} + \|e\|_{K,k}).
\end{gather*}
\end{itemize}
Finally, $E$ is said to be \emph{boundedly stable} if for every $N \in \N$ and every bounded set $B$ in $E_N$ there is $M \geq N$ such that for every $K \geq M$ the spaces $E_M$ and $E_K$ induce the same topology on $B$. Grothendieck's factorization theorem (see e.g.\ \cite[p.\ 225]{Kothe}) implies that all of these conditions do not depend on the defining inductive spectrum of $E$. This justifies calling an $(LF)$-space boundedly retractive, etc., if one (and thus all) of its defining inductive spectra has this property. 
The following result shall be of crucial importance to us.
\begin{theorem} \label{reg-cond}
Let $E$ be an $(LF)$-space. Consider the following conditions:
\begin{itemize}
\item[$(i)$] $E$ is boundedly retractive.
\item[$(ii)$] $E$ is complete.
\item[$(iii)$] $E$ is regular.
\item[$(iv)$] $E$  is $\beta$-regular.
\item[$(v)$] $E$ satisfies $(wQ)$.
\end{itemize}
Then, $(i) \Rightarrow (ii) \Rightarrow (iii) \Rightarrow (iv) \Rightarrow (v)$. Moreover, if $E$ is boundedly stable, then $(v) \Rightarrow (i)$.
\end{theorem}
We refer to \cite[Sect.\ 6]{Wengenroth} (see also \cite{Wengenroth-96} and the references therein) for the proof of and more information on Theorem \ref{reg-cond}.

\subsection{Regularity properties} We now introduce the weighted $(LF)$-spaces of smooth functions that we shall be concerned with.
Let $v$ be a non-negative function on $\R^d$ and let $n \in \N$. We define $\mathcal{B}^n_v(\R^d)$ as the seminormed space consisting of all $\varphi \in C^n(\R^d)$ such that
$$
\|\varphi\|_{v,n} := \max_{|\alpha| \leq n} \sup_{x \in \R^d} |\partial^{\alpha}\varphi(x)|v(x) < \infty. 
$$
The closure of $\mathcal{D}(\R^d)$ in $\mathcal{B}^n_v(\R^d)$ is denoted by ${\dot{\mathcal{B}}}^n_v(\R^d)$. Clearly, the latter space consists of all $\varphi \in C^n(\R^d)$ such that
$$
\lim_{|x| \to \infty}|\partial^{\alpha}\varphi(x)|v(x) = 0
$$
for all $|\alpha| \leq n$ and we also endow it with the seminorm $\|\:\cdot \:\|_{v,n}$. If $v$ is positive and $v^{-1}$ is locally bounded, $\mathcal{B}^n_v(\R^d)$ and ${\dot{\mathcal{B}}}^n_v(\R^d)$ are Banach spaces. These requirements are fulfilled if $v$ is positive and continuous.
 For $n=0$ we simply write $\mathcal{B}^0_v(\R^d) = Cv(\R^d)$, ${\dot{\mathcal{B}}}^0_v(\R^d) = C(v)_0(\R^d)$ and $\|\:\cdot \:\|_{v,0} = \|\:\cdot \:\|_v$. 
Furthermore, we set
$$
\mathcal{B}_v(\R^d):= \varprojlim_{n \in \N} \mathcal{B}^n_{v}(\R^d), \qquad \dot{\mathcal{B}}_v(\R^d) := \varprojlim_{n \in \N} {\dot{\mathcal{B}}}^n_{v}(\R^d).
$$
A (pointwise) decreasing sequence $\mathcal{V} = (v_{N})_{N \in \N}$ of positive continuous functions on $\R^d$ is called a \emph{decreasing weight system}. We define the following $(LF)$-spaces
$$
\mathcal{B}_{\mathcal{V}}(\R^d) := \varinjlim_{N \in \N} \mathcal{B}_{v_N}(\R^d), \qquad \dot{\mathcal{B}}_{\mathcal{V}}(\R^d) := \varinjlim_{N \in \N} \dot{\mathcal{B}}_{v_N}(\R^d).
$$
\begin{remark}\label{V-remark}
If $\mathcal{V} = (v_N)_N$ satisfies condition $(V)$ (cf.\ \cite[p.\ 114]{B-M-S}), i.e.,
$$
\forall N \in \N \, \exists  M > N: \lim_{|x| \to \infty} v_M(x)/v_N(x) = 0,
$$
then $\mathcal{B}_{\mathcal{V}}(\R^d) = \dot{\mathcal{B}}_{\mathcal{V}}(\R^d)$ and this space is an $(LFS)$-space (= inductive limit of $(FS)$-spaces).
\end{remark}
 We shall often need to impose the following mild condition on $\mathcal{V}$:
\begin{equation}
\forall N \in \N\, \exists \widetilde{N} \geq N \, : \, g_{N,\widetilde{N}} = \sup_{x \in \R^d}\frac{v_{\widetilde{N}}(x + \: \cdot \:)}{v_N(x)} \in L^\infty_{\operatorname{loc}}(\R^d).
\label{locally-bounded-decreasing}
\end{equation}

We are ready to discuss the regularity properties of the $(LF)$-spaces $\mathcal{B}_{\mathcal{V}}(\R^d)$ and $\dot{\mathcal{B}}_{\mathcal{V}}(\R^d)$.  We need the following definition.
\begin{definition} 
A decreasing weight system  $\mathcal{V} =(v_N)_N$ is said to satisfy $(\Omega)$ if
\begin{gather*}
\forall N \in \N \, \exists M \geq N \, \forall K \geq M \, \exists \theta \in (0,1) \, \exists C > 0 \, \forall x \in \R^d: \\
v_M(x) \leq C {v_N(x)}^{1-\theta}{v_K(x)}^{\theta}.
\end{gather*}
\end{definition}
\begin{theorem}\label{completeness-ind-lim-smooth}
Let $\mathcal{V} = (v_N)_{N}$ be a decreasing weight system satisfying \eqref{locally-bounded-decreasing}. Then, the following statements are equivalent:
\begin{itemize}
\item[$(i)$] $\mathcal{V}$ satisfies $(\Omega)$.
\item[$(ii)$] $\mathcal{B}_\mathcal{V}(\R^d)$ is boundedly retractive.
\item[$(iii)$] $\mathcal{B}_\mathcal{V}(\R^d)$ satisfies $(wQ)$.
\item[$(ii)'$] $\dot{\mathcal{B}}_\mathcal{V}(\R^d)$ is boundedly retractive.
\item[$(iii)'$] $\dot{\mathcal{B}}_\mathcal{V}(\R^d)$ satisfies $(wQ)$.
\end{itemize}
\end{theorem}
The proof of Theorem \ref{completeness-ind-lim-smooth} is based on the ensuing three lemmas. 
\begin{lemma} \label{boundedly-stable-smooth}
Let $\mathcal{V} = (v_N)_{N}$ be a decreasing weight system satisfying $(\Omega)$. Then, $\mathcal{B}_\mathcal{V}(\R^d)$ and $\dot{\mathcal{B}}_\mathcal{V}(\R^d)$ are boundedly stable.
\end{lemma}
\begin{proof}
Notice that $\dot{\mathcal{B}}_\mathcal{V}(\R^d)$ is boundedly stable if $\mathcal{B}_\mathcal{V}(\R^d)$ is so. Hence, it suffices to show that $\mathcal{B}_\mathcal{V}(\R^d)$ is boundedly stable. Let $N \in \N$ be arbitrary and choose $M \geq N$ according to $(\Omega)$. We shall show that for all $K \geq M$ the spaces $\mathcal{B}_{v_K}(\R^d)$ and $\mathcal{B}_{v_M}(\R^d)
$ induce the same topology on the bounded sets $B$ of  $\mathcal{B}_{v_N}(\R^d)$. We only need to prove that the topology induced by $\mathcal{B}_{v_K}(\R^d)$ is finer than the 
one induced by $\mathcal{B}_{v_M}(\R^d)$. Consider the basis of neighborhoods of $0$ in $\mathcal{B}_{v_M}(\R^d)$ given by 
$$
U(n,\varepsilon) =  \{ \varphi \in \mathcal{B}_{v_M}(\R^d) \, : \, \|\varphi \|_{v_M,n} \leq \varepsilon \}, \qquad n \in \N,\varepsilon > 0.
$$
Let $n \in \N$ and $\varepsilon > 0$ be arbitrary. Choose $\theta \in (0,1)$ and $C > 0$ such that $v_M(x) \leq C {v_N(x)}^{1-\theta}v_K(x)^{\theta}$ for all $x \in \R^d$. Set  $\delta = (\varepsilon/C)^{1/\theta}(\sup_{\varphi \in B} \|\varphi\|_{v_N,n})^{-(1-\theta)/\theta}$ and $V = \{ \varphi \in \mathcal{B}_{v_K}(\R^d)  \, : \, \| \varphi \|_{v_K,n} \leq \delta \}$. We claim that $V \cap B \subseteq U(n,\varepsilon)$. Indeed, for $\varphi \in V \cap B$ we have that
\begin{align*}
\| \varphi \|_{v_M,n} &=\max_{|\alpha| \leq n} \sup_{x \in \R^d} |\partial^{\alpha}\varphi(x)|v_M(x) \leq C\max_{|\alpha| \leq n}\sup_{x \in \R^d} (|\partial^{\alpha}\varphi(x)| {v_N}(x))^{1-\theta}(|\partial^{\alpha}\varphi(x)|v_K(x))^{\theta} \\
& \leq C\| \varphi\|^{1-\theta}_{v_N,n}\| \varphi\|^\theta_{v_K,n} \leq \varepsilon.
\qedhere
\end{align*}
\end{proof}
\begin{lemma}\label{gorny-wQ}
Let $\mathcal{V} = (v_N)_{N}$ be a decreasing weight system satisfying \eqref{locally-bounded-decreasing} and $(\Omega)$. Then, 
\begin{gather*}
\forall N \in \N \, \exists M \geq N \, \forall K \geq M \, \forall m \in \N \, \exists k \in \N \, \exists \rho \in (0,1) \, \exists C > 0\, \forall \varphi \in \mathcal{B}_{v_N}(\R^d) : \\
\|\varphi\|_{v_M,m} \leq C(\|\varphi\|^{1-\rho}_{v_N} \|\varphi\|^{\rho}_{v_K,k}).
\end{gather*}
\end{lemma} 
\begin{proof}
The proof is based on the following  Landau-Kolmogorov type inequality for a cube (also known as Gorny's inequality in dimension one), shown in \cite[Example 3.13]{Frerick}:  For all $m,k \in \N$ with $m \leq k$ there is $C > 0$ such that
$$
\|f\|_{[-1,1]^d,m} \leq C\|f\|^{1-(m/k)}_{[-1,1]^d}\|f\|^{m/k}_{[-1,1]^d,k}, \qquad f \in C^\infty([-1,1]^d).
$$
Let $N \in \N$ be arbitrary and choose $\widetilde{N} \geq N$ as in \eqref{locally-bounded-decreasing}. Next, choose $M \geq \widetilde{N}$ according to $(\Omega)$. Let $K \geq M$ and $m \in \N$ be arbitrary. Choose $\widetilde{K} \geq K$ as in \eqref{locally-bounded-decreasing}. There are $\theta \in (0,1)$ and $C' > 0$ such that $v_M(x) \leq C'v
_{\widetilde{N}}(x)^{1-\theta}v_{\widetilde{K}}(x)^{\theta}$ for all $x \in \R^d$. Finally, let $k \in \N$ be so large that $\rho = m/k \leq \theta$. Notice that for $0 \leq a \leq b $ it holds that $b^{1-\theta}a^\theta \leq b^{1-\rho}a^\rho$. Hence, for $\varphi \in \mathcal{B}_{v_N}(\R^d)$, it 
holds that
\begin{align*}
\|\varphi\|_{v_M,m} &\leq \sup_{x \in \R^d} \|T_{-x}\varphi\|_{[-1,1]^d,m}v_M(x) \\
&\leq CC' \sup_{x \in \R^d} \|T_{-x}\varphi\|^{1-(m/k)}_{[-1,1]^d}\|T_{-x}\varphi\|^{m/k}_{[-1,1]^d,k}v_{\widetilde{N}}(x)^{1-\theta}v_{\widetilde{K}}(x)^{\theta} \\
&\leq CC'(\sup_{x \in \R^d} \|T_{-x}\varphi\|_{[-1,1]^d}v_{\widetilde{N}}(x))^{1-\rho}(\sup_{x \in \R^d} \|T_{-x}\varphi\|_{[-1,1]^d,k}v_{\widetilde{K}}(x))^\rho.
\end{align*}
The result  follows from the fact that
$$
\sup_{x \in \R^d} \|T_{-x}\varphi\|_{[-1,1]^d,j}v_{\widetilde{J}}(x) \leq \|\check{g}_{J,\widetilde{J}}\|_{[-1,1]^d} \|\varphi\|_{v_J,j}, 
$$
for all $j,J \in \N$ and $\varphi \in  \mathcal{B}_{v_J}(\R^d)$. 
\end{proof}
In the next lemma, $s$ denotes the Fr\'echet space of rapidly decreasing sequences, that is, the space consisting of all sequences $(c_j)_{j \in \N} \in \C^\N$ such that $\sup_{j \in \N}|c_j| j^k < \infty$ for all $k \in \N$.
\begin{lemma}\label{suff-Omega}
Let $\mathcal{V} = (v_N)_N$ be a decreasing weight system and let $E$ be a Fr\'echet space such that $ E \cong s$ topologically. Let $(\| \: \cdot \; \|_n)_{n \in \N}$ be a fundamental increasing sequence of seminorms  for $E$ and  suppose that the pair $(E, \mathcal{V})$ satisfies $(S_2)^*$ (cf.\ \cite[Lemma 2.1]{Albanese}), i.e.,
\begin{gather*}
\forall N \in \N \, \exists M \geq N \, \exists n \in \N \, \forall K \geq M \, \forall m \in \N \, \exists k \in \N \, \exists C > 0\, \forall e \in E \,  \forall x \in \R^d: \\
v_M(x)\|e\|_m \leq C(v_N(x)\|e\|_n + v_K(x)\|e\|_k).
\end{gather*}
Then, $\mathcal{V}$ satisfies $(\Omega)$.
\end{lemma}
\begin{proof}
If $E$ and $F$ are topologically isomorphic Fr\'echet spaces, then $(E, \mathcal{V})$ satisfies $(S_2)^*$ if and only if  $(F, \mathcal{V})$ does so. Hence, it suffices to show that 
$\mathcal{V}$ satisfies $(\Omega)$ provided that $(s,\mathcal{V})$ satisfies $(S_2)^*$. This is essentially shown in  \cite[Thm.~4.1]{Vogt}, but we repeat the argument here for the sake of completeness. Let $N \in \N$ be arbitrary and choose $M \geq N$ as in $(S_2)^\ast$. Let $K \geq M$ be arbitrary. Choose $k \in \N$ according to $(S_2)^\ast$ for $m = n+1$. We may assume that $k > m$. By applying $(S_2)^\ast$ to the unit vectors in $s$, we obtain that there is $C \geq 1$ such that
$$
v_M(x) \leq C\left(\frac{v_N(x)}{j} + v_K(x)j^{(k-m)}\right)
$$
for all $x \in \R^d$ and $j \in \Z_+$, which implies that
$$
v_M(x) \leq 2^{k-m}C\left(\frac{v_N(x)}{r} + v_K(x)r^{(k-m)}\right)
$$
for all $x \in \R^d$ and $r >0$. The result  follows by minimizing the right-hand side for $r > 0$ (with $x \in \R^d$ fixed).
\end{proof}
\begin{proof}[Proof of Theorem \ref{completeness-ind-lim-smooth}]
The implications $(ii)  \Rightarrow (iii)$ and $(ii)'  \Rightarrow (iii)'$ hold for general $(LF)$-spaces by Theorem \ref{reg-cond}, while $(iii) \Rightarrow (iii)'$ is obvious. 
 In view of Lemmas \ref{boundedly-stable-smooth} and \ref{gorny-wQ}, $(i) \Rightarrow (ii)$ and $(i) \Rightarrow (ii)'$ follow from another application of Theorem \ref{reg-cond}.  Finally, we show $(iii)' \Rightarrow (i)$. By Lemma \ref{suff-Omega} and the fact that $\mathcal{D}_{[-1,1]^d} \cong s$ topologically \cite[Prop.\ 31.12]{M-V}, it suffices to show that  $(\mathcal{D}_{[-1,1]^d}, \mathcal{V})$ satisfies $(S_2)^\ast$. Let $N \in \N$ be arbitrary and choose $\widetilde{N} \geq N$ as in \eqref{locally-bounded-decreasing}. Next, choose $M \geq \widetilde{N}$ and $n \in \N$ according to $(wQ)$. Pick $\widetilde{M} \geq M$ as in \eqref{locally-bounded-decreasing}. We shall show $(S_2)^\ast$ for $\widetilde{M}$ and $n$. Let $K  \geq \widetilde{M}$ and $m \in \N$ be arbitrary. Choose $\widetilde{K} \geq K$ as in \eqref{locally-bounded-decreasing}. By $(wQ)$, there are $k \in \N$ and $C > 0$ such that
$$
\|\varphi\|_{v_M,m} \leq C\Big(\|\varphi\|_{v_{\widetilde{N}},n} + \|\varphi\|_{{v_{\widetilde{K}},k}}\Big), \qquad \varphi \in  \dot{\mathcal{B}}_{v_N}(\R^d).
$$
Let $x \in \R^d$ and $\varphi \in \mathcal{D}_{[-1,1]^d}$ be arbitrary. For all $j,J \in \N$ we have that $T_x\varphi \in \mathcal{D}(\R^d) \subset \dot{\mathcal{B}}_{v_J}(\R^d)$ and 
$$
\|T_x\varphi\|_{v_{\widetilde{J}},j} \leq \|g_{J,\widetilde{J}}\|_{[-1,1]^d} v_J(x) \|\varphi\|_{[-1,1]^d,j}
$$
and
$$
v_{\widetilde{J}}(x) \|\varphi\|_{[-1,1]^d,j}\leq \|\check{g}_{J,\widetilde{J}}\|_{[-1,1]^d} \|T_x\varphi\|_{v_J,j}. 
$$
Hence, 
\begin{align*}
v_{\widetilde{M}}(x) \|\varphi\|_{[-1,1]^d,m} &\leq \|\check{g}_{M,\widetilde{M}}\|_{[-1,1]^d}\|T_x\varphi\|_{v_M,m}  \\
&\leq C\|\check{g}_{M,\widetilde{M}}\|_{[-1,1]^d}(\|T_x\varphi\|_{v_{\widetilde{N}},n} + \|T_x\varphi\|_{{v_{\widetilde{K}},k}})\\
&\leq C'(v_N(x)\|\varphi\|_{[-1,1]^d,n}  + v_K(x)\|\varphi\|_{[-1,1]^d,k}),
\end{align*}
where $C' = C\|\check{g}_{M,\widetilde{M}}\|_{[-1,1]^d}\max\{\|g_{N,\widetilde{N}}\|_{[-1,1]^d},\|g_{K,\widetilde{K}}\|_{[-1,1]^d}\}$.
\end{proof}

\subsection{Characterization via the STFT} We now turn our attention to the mapping properties of the STFT on the spaces $\mathcal{B}_{\mathcal{V}}(\R^d)$ and $\dot{\mathcal{B}}_{\mathcal{V}}(\R^d)$. The following two technical lemmas are needed.

\begin{lemma}\label{STFT-test-smooth}
Let $\psi \in \mathcal{D}(\R^d)$ and let $w$ and $v$ be positive continuous functions on $\R^d$ such that
\begin{equation}
g = \sup_{x \in \R^d} \frac{v(x + \:\cdot\:)}{w(x)} \in L^\infty_{\operatorname{loc}}(\R^d).
\label{locally-bounded-cond}
\end{equation}
Then, the mappings
$$
V_\psi: \mathcal{B}^n_{w}(\R^d) \rightarrow C(v\otimes (1 + |\:\cdot\:|)^n)(\R^{2d}_{x,\xi})
$$ 
and
$$
V_\psi: {\dot{\mathcal{B}}}^n_{w}(\R^d) \rightarrow C(v\otimes (1 + |\:\cdot\:|)^n)_0(\R^{2d}_{x,\xi})
$$
are well-defined and continuous.
\end{lemma}
\begin{proof}
Set $K = \operatorname{supp} \psi$ and let $\varphi \in \mathcal{B}^n_{w}(\R^d)$ be arbitrary. For all $(x,\xi) \in \R^{2d}$ it holds that
\begin{align*}
&|V_\psi\varphi(x,\xi)| v(x) (1 + |\xi|)^n \leq (1+ \sqrt{d})^n \max_{|\alpha| \leq n}|\xi^\alpha V_\psi\varphi(x,\xi)|v(x) \\
&\leq (1+ \sqrt{d})^n \max_{|\alpha| \leq n}(2\pi)^{-|\alpha|} \sum_{\beta \leq \alpha} \binom{\alpha}{\beta} \int_{K} |\partial^{\beta}\varphi(t+x)|w(t+x)\check{g}(t)|\partial^{\alpha-\beta}\psi(t)| \dt \\
&\leq (1+ \sqrt{d})^n|K| \|\psi\|_{K,n} \|\check{g}\|_{K} \max_{|\alpha| \leq n}\sup_{t \in K}|\partial^{\alpha}\varphi(t+x)|w(t+x),
\end{align*}
which shows the continuity of $V_\psi$ on $\mathcal{B}^n_{w}(\R^d)$. 
We still need to show that $V_\psi({\dot{\mathcal{B}}}^n_{w}(\R^d)) \subseteq C(v\otimes (1 + |\:\cdot\:|)^n)_0(\R^{2d})$. Let $\varphi \in {\dot{\mathcal{B}}}^n_{w}(\R^d)$ be arbitrary. The above inequality shows that
$$
\lim_{|x| \to \infty} \sup_{\xi \in \R^d} |V_\psi\varphi(x,\xi)| v(x) (1 + |\xi|)^n = 0.
$$
Hence, we only need to prove that
$$
\lim_{|\xi| \to \infty} \sup_{x \in L} |V_\psi\varphi(x,\xi)| v(x) (1 + |\xi|)^n = 0
$$
for all compacts $L \Subset \R^d$. Since
\begin{align*}
&\sup_{x \in L} |V_\psi\varphi(x,\xi)| v(x) (1 + |\xi|)^n \\
&\leq  (1+ \sqrt{d})^n\|v\|_{L}  \sup_{x \in L}\max_{|\alpha| \leq n}(2\pi)^{-|\alpha|} \sum_{\beta \leq \alpha}\binom{\alpha}{\beta} |\mathcal{F}(\partial^{\beta}\varphi T_x \partial^{\alpha-\beta}\overline{\psi})(\xi)|,
\end{align*}
it suffices to show that
$$
\lim_{|\xi| \to \infty} \sup_{x \in L} |\mathcal{F}(f T_x \chi)(\xi)| = \lim_{|\xi| \to \infty} \sup_{x \in L}|\langle e^{-2\pi i t \xi}, f(t) T_x \chi(t) \rangle_{(L^\infty,L^1)}| = 0
$$
for all $f \in C(\R^d)$ and $\chi \in \mathcal{D}(\R^d)$.  Since the set $\{ e^{-2\pi i t  \xi} \, : \, \xi \in \R^d\}$ is bounded in $L^\infty(\R^d_t)$ and  $\lim_{|\xi| \rightarrow \infty} e^{-2\pi i t \xi} = 0$ in $L^\infty(\R^d_t)$ endowed with the weak-$\ast$ topology (Riemann-Lebesgue lemma), we obtain that $\lim_{|\xi| \rightarrow \infty} e^{-2\pi i t \xi} = 0$ on compact subsets of $L^1(\R^d)$. The result  follows by observing that the set $\{ f T_x \chi \, : \, x \in L\}$ is compact in $L^1(\R^d)$.\end{proof}
\begin{lemma}\label{double-int-test-smooth}
Let $\psi \in \mathcal{D}(\R^d)$ and let $w$ and $v$ be positive continuous functions on $\R^d$ for which \eqref{locally-bounded-cond} holds. Then, the mappings
$$
V^\ast_\psi: C(w\otimes (1+|\:\cdot\:|)^{n+d+1})(\R^{2d}_{x,\xi}) \rightarrow \mathcal{B}^{n}_v(\R^d)
$$
and
$$
V^\ast_\psi: C(w\otimes (1+|\:\cdot\:|)^{n+d+1})_0(\R^{2d}_{x,\xi}) \rightarrow {\dot{\mathcal{B}}}^{n}_v(\R^d)
$$
are well-defined and continuous.
\end{lemma}
\begin{proof}
Set $K = \operatorname{supp} \psi$ and let $F \in C(w\otimes (1+|\: \cdot \:|)^{n+d+1})(\R^{2d})$ be arbitrary.  For all $t \in \R^d$ and $|\alpha| \leq n$ it holds that
\begin{align*}
|\partial^\alpha V^\ast_\psi F(t)|v(t) &\leq  \sum_{\beta \leq \alpha} \binom{\alpha}{\beta} v(t)\int \int_{\R^{2d}} |F(x,\xi)|(2\pi|\xi|)^{|\beta|}|\partial^{\alpha-\beta}\psi(t-x)| \dx \dxi \\
&\leq (2\pi)^n|K|\|\psi\|_{K,n}\|g\|_{K} \int_{\R^d}\sup_{x \in K} \frac{|F(t-x,\xi)|w(t-x)(1+|\xi|)^{n+d+1}}{(1+|\xi|)^{d+1}} \dxi,
\end{align*}
which shows the continuity of $V^\ast_\psi$ on  $C(w\otimes (1+|\: \cdot \:|)^{n+d+1})(\R^{2d})$. The fact that $V^\ast_\psi(C(w\otimes (1+|\: \cdot \:|)^{n+d+1})_0(\R^{2d})) \subseteq {\dot{\mathcal{B}}}^{n}_v(\R^d)$ follows from the above inequality and Lebesgue's dominated convergence theorem.
\end{proof}
Given a decreasing weight system $\mathcal{V} = (v_N)_N$, we define the following  $(LF)$-spaces of continuous functions
$$
\mathcal{V}_{\operatorname{pol}}C(\R^{2d}) := \varinjlim_{N \in \N} \varprojlim_{n \in \N}  C(v_N\otimes (1+|\: \cdot \:|)^{n})(\R^{2d})
$$
and
$$
\mathcal{V}_{\operatorname{pol},0}C(\R^{2d}) := \varinjlim_{N \in \N} \varprojlim_{n \in \N}  C(v_N\otimes (1+|\: \cdot \:|)^{n})_0(\R^{2d}).
$$
Lemmas \ref{STFT-test-smooth} and \ref{double-int-test-smooth} together with \eqref{reconstruction-D-dual} directly imply the ensuing two results.
\begin{proposition}\label{STFT-test-char-smooth}
Let $\psi \in \mathcal{D}(\R^d)$ and let $\mathcal{V} = (v_N)_N$ be a decreasing weight system satisfying \eqref{locally-bounded-decreasing}. Then, the mappings
$$
V_\psi: \mathcal{B}_{\mathcal{V}}(\R^d) \rightarrow \mathcal{V}_{\operatorname{pol}}C(\R^{2d}_{x,\xi}) 
$$ 
and 
$$
V^\ast_\psi: \mathcal{V}_{\operatorname{pol}}C(\R^{2d}_{x,\xi}) \rightarrow \mathcal{B}_{\mathcal{V}}(\R^d)
$$
are well-defined and continuous. Moreover, if $\psi \neq 0$ and $\gamma \in \mathcal{D}(\R^d)$ is a synthesis window for $\psi$, then
$$
\frac{1}{(\gamma, \psi)_{L^2}} V^\ast_\gamma \circ V_\psi = \operatorname{id}_{\mathcal{B}_{\mathcal{V}}(\R^d)}.
$$
\end{proposition}
\begin{proposition}\label{STFT-test-char-smooth-1}
Let $\psi \in \mathcal{D}(\R^d)$ and let $\mathcal{V} = (v_N)_N$ be a decreasing weight system satisfying \eqref{locally-bounded-decreasing}. Then, the mappings
$$
V_\psi: \dot{\mathcal{B}}_{\mathcal{V}}(\R^d) \rightarrow \mathcal{V}_{\operatorname{pol},0}C(\R^{2d}_{x,\xi})
$$ 
and 
$$
V^\ast_\psi: \mathcal{V}_{\operatorname{pol},0}C(\R^{2d}_{x,\xi}) \rightarrow \dot{\mathcal{B}}_{\mathcal{V}}(\R^d)
$$
are well-defined and continuous. 
\end{proposition}
Finally, we combine Proposition \ref{STFT-test-char-smooth} with a result concerning the projective description of weighted $(LF)$-spaces of continuous functions from \cite{B-B} to give an explicit system of seminorms generating the topology of $\mathcal{B}_{\mathcal{V}}(\R^d)$ in case $\mathcal{V}$ satisfies $(\Omega)$; this will be of vital importance later on.

 A double sequence $\mathcal{U} = (u_{N,n})_{(N,n) \in \N^2}$ of positive continuous functions on $\R^d$ is called a \emph{weight system} if $u_{N,n}(x) \geq u_{N+1,n}(x)$ and $u_{N,n}(x) \leq u_{N,n+1}(x)$ for all $N,n \in \N$ and $x \in \R^d$. We define the following $(LF)$-space
$$
\mathcal{U}C(\R^d) := \varinjlim_{N \in \N} \varprojlim_{n \in \N} Cu_{N,n}(\R^d).
$$
The \emph{maximal Nachbin family associated with $\mathcal{U}$}, denoted by $ \overline{U}=\overline{U}(\mathcal{U})$, is given by the space consisting of all non-negative upper semicontinuous functions $u$ on $\R^d$ for which there is a sequence of natural numbers $(n_N)_{N \in \N}$ such that $\sup_{x \in \R^d} u(x)/u_{N,n_N}(x) < \infty$ for all $N \in \N$.
We define $C\overline{U}(\R^d)$ as the space consisting of all $f \in C(\R^d)$ such that $\| f \|_u < \infty$ for all $u \in \overline{U}$ and endow it with the locally convex topology generated by the system of seminorms $\{ \| \: \cdot \: \|_u \, : \, u \in \overline{U}\}$. Clearly, $\mathcal{U}C(\R^d)$ is continuously included in $C\overline{U}(\R^d)$. The problem of  projective description in this context is to find conditions on $\mathcal{U}$ which ensure that $\mathcal{U}C(\R^d)$ and $C\overline{U}(\R^d)$ coincide algebraically and/or topologically. This problem was thoroughly studied by Bierstedt and Bonet in \cite{B-B}. We shall use the following result of these authors\footnote{In fact, this result holds for weight systems defined on a general locally compact Hausdorff space.}:

\begin{theorem}[{\cite[Cor.\ 5, p.\ 42]{B-B}}]
\label{proj-desc-cont-LF}

Let $\mathcal{U} = (u_{N,n})_{N,n}$ be a weight system satisfying condition $(Q)$, i.e.,
\begin{gather*}
\forall N \in \N \, \exists M \geq N \, \exists n \in \N \, \forall K \geq M \, \forall m \in \N \, \forall \varepsilon > 0 \, \exists k \in \N \, \exists C > 0\, \forall x \in \R^d: \\
u_{M,m}(x) \leq \varepsilon u_{N,n}(x) + Cu_{K,k}(x).
\end{gather*}
Then, $\mathcal{U}C(\R^d) = C\overline{U}(\R^d)$ topologically.
\end{theorem}
We are now able to show the following result:
\begin{corollary}\label{proj-desc-smooth-functions}
Let $\mathcal{V} = (v_N)_N$ be a decreasing weight system satisfying $(\Omega)$. Then, the weight system 
$$
\mathcal{V} _{\operatorname{pol}} = (v_N\otimes (1+|\: \cdot \:|)^{n})_{N,n}
$$
(defined on $\R^{2d}$) satisfies $(Q)$. Consequently, if $\mathcal{V}$ satisfies \eqref{locally-bounded-decreasing} and $\psi \in \mathcal{D}(\R^d) \backslash \{0\}$ is fixed,  $f \in \mathcal{D}'(\R^d)$ belongs to $\mathcal{B}_{\mathcal{V}}(\R^d)$ if and only if  $\|V_\psi f \|_u < \infty$ for all $u \in  \overline{U}(\mathcal{V} _{\operatorname{pol}})$. Moreover, the topology of $\mathcal{B}_{\mathcal{V}}(\R^d)$ is generated by the system of seminorms $\{ \|V_\psi ( \: \cdot \: ) \|_u \, : \, u \in \overline{U}(\mathcal{V} _{\operatorname{pol}}) \}$.
\end{corollary}
\begin{proof}
The second part follows from the first one, Theorem \ref{proj-desc-cont-LF}, and Proposition \ref{STFT-test-char-smooth}. We now prove that  $\mathcal{V} _{\operatorname{pol}}$ satisfies $(Q)$ if $\mathcal{V}$ satisfies $(\Omega)$ by employing the same idea as in \cite[Thm.\ 5.1]{Vogt}. We shall show $(Q)$ with $n = 0$.  Let $N \in \N$ be arbitrary and choose $M \geq N$ according to $(\Omega)$. Next, let $K \geq M$ and $m \in \N$ be arbitrary. Condition $(\Omega)$ implies that there are $\theta \in (0,1)$ and $C > 0$ such that $v_M(x) \leq C {v_N(x)}^{1-\theta}{v_K(x)}^{\theta}$ for all $x \in \R^d$. Choose  $k \in \N$ so large that $k\theta \geq m$.
Let $\varepsilon > 0$ be arbitrary and set $C' = (C\varepsilon^{\theta-1})^{1/\theta}$. We claim that
$$
v_{M}(x)(1+|\xi|)^m \leq \varepsilon v_{N}(x)  + C'v_{K}(x)(1+|\xi|)^k, \qquad (x,\xi) \in \R^{2d}. 
$$
Let $(x,\xi) \in \R^{2d}$ be arbitrary. If $v_{M}(x)(1+|\xi|)^m \leq \varepsilon v_{N}(x)$, we are done. So, we may assume that $v_{M}(x)(1+|\xi|)^m > \varepsilon v_{N}(x)$. Hence,
\begin{align*}
v_{M}(x)(1+|\xi|)^m &\leq  C{v_N(x)}^{1-\theta}(v_K(x)(1+|\xi|)^k))^{\theta} \\
&\leq C\varepsilon^{\theta-1} (v_M(x)(1+|\xi|)^m)^{1-\theta}(v_K(x)(1+|\xi|)^k)^{\theta},
\end{align*}
and thus $v_{M}(x)(1+|\xi|)^m \leq C'v_{K}(x)(1+|\xi|)^k$.
\end{proof}

\section{On a class of weighted $L^1$ convolutor spaces}\label{L1-convolutors}
In this section, we introduce a class of weighted $L^1$ convolutor spaces $\mathcal{O}'_C(\mathcal{D}, L^1_{\mathcal{W}})$ defined via an increasing sequence of positive continuous functions $\mathcal{W}$. In the main theorem of this section we characterize the increasing weight systems $\mathcal{W}$ for which the space $\mathcal{O}'_C(\mathcal{D}, L^1_{\mathcal{W}})$ is ultrabornological. To achieve this goal, we first study various structural and topological properties of these spaces; most notably, we determine a topological predual and give an explicit description of the dual of $\mathcal{O}'_C(\mathcal{D}, L^1_{\mathcal{W}})$. Since the proofs of the latter results will heavily rely on the mapping properties of the STFT on $\mathcal{O}'_C(\mathcal{D}, L^1_{\mathcal{W}})$, we start with a discussion of the STFT on this space.

Let $w$ be a positive measurable function on $\R^d$. We define $L^1_w(\R^d)$ as the Banach space consisting of all measurable functions $f$ on $\R^d$ such that
$$
\|f\|_{L^1_w} := \int_{\R^d}|f(x)|w(x) \dx < \infty.
$$
A pointwise increasing  sequence $\mathcal{W} = (w_N)_{N \in \N}$  of positive continuous functions on $\R^d$ is called an \emph{increasing weight system}. Set
$$
L^1_\mathcal{W}(\R^d) := \varprojlim_{N \in \N} L^1_{w_N}(\R^d).
$$
We are interested in the following convolutor spaces
$$
\mathcal{O}'_C(\mathcal{D},L^1_{\mathcal{W}}) := \{ f \in \mathcal{D}'(\R^d) \, : \, f \ast \varphi \in L^1_\mathcal{W}(\R^d) \mbox{ for all } \varphi \in \mathcal{D}(\R^d) \}.
$$
For $f \in \mathcal{O}'_C(\mathcal{D},L^1_{\mathcal{W}})$ fixed, the mapping $\mathcal{D}(\R^d) \rightarrow L^1_{\mathcal{W}}(\R^d),$  $\varphi \mapsto f \ast \varphi$ is continuous, as follows from the continuity of the mapping $\mathcal{D}(\R^d) \rightarrow \mathcal{E}(\R^d),$ $\varphi \mapsto f \ast \varphi$ and the closed graph theorem. We endow 
$\mathcal{O}'_C(\mathcal{D},L^1_{\mathcal{W}})$ with the initial topology with respect to the mapping
$$
\mathcal{O}'_C(\mathcal{D},L^1_{\mathcal{W}}) \rightarrow L_b(\mathcal{D}(\R^d),L^1_{\mathcal{W}}(\R^d)), \quad f \mapsto (\varphi \mapsto f \ast \varphi).
$$
Next, we define $\mathcal{D}_{L^1_{\mathcal{W}}}(\R^d)$ as the Fr\'echet space consisting of all $\varphi \in C^\infty(\R^d)$ such that 
$
\max_{|\alpha| \leq n} \| \partial^\alpha \varphi\|_{L^1_{w_N}} < \infty
$
for all $N,n \in \N$.  Set 
$$
\mathcal{O}'_C(\mathcal{D},\mathcal{D}_{L^1_{\mathcal{W}}}) := \{ f \in \mathcal{D}'(\R^d) \, : \, f \ast \varphi \in \mathcal{D}_{L^1_{\mathcal{W}}}(\R^d)\mbox{ for all } \varphi \in \mathcal{D}(\R^d) \}
$$
and endow this space with the initial topology with respect to the mapping
$$
\mathcal{O}'_C(\mathcal{D}, \mathcal{D}_{L^1_{\mathcal{W}}}) \rightarrow L_b(\mathcal{D}(\R^d), \mathcal{D}_{L^1_{\mathcal{W}}}(\R^d)), \quad f \mapsto (\varphi \mapsto f \ast \varphi).
$$
The next lemma follows from the fact that $\partial^\alpha (f \ast \varphi) = f \ast \partial^\alpha \varphi$ for all $f \in \mathcal{D}'(\R^d)$ and $\varphi \in \mathcal{D}(\R^d)$.
\begin{lemma}\label{new} Let  $\mathcal{W} = (w_N)_N$ be an increasing weight system. Then,
$$
\mathcal{O}'_C(\mathcal{D},L^1_{\mathcal{W}}) = \mathcal{O}'_C(\mathcal{D}, \mathcal{D}_{L^1_{\mathcal{W}}})
$$ 
topologically. 
\end{lemma}

Finally, we introduce the following condition on $\mathcal{W} = (w_N)_N$ (cf.\ \eqref{locally-bounded-decreasing}):
\begin{equation}
\forall N \in \N\, \exists \widetilde{N} \geq N \, : \, h_{N,\widetilde{N}} = \sup_{x \in \R^d}\frac{w_{N}(x + \: \cdot \:)}{w_{\widetilde{N}}(x)} \in L^\infty_{\operatorname{loc}}(\R^d).
\label{locally-bounded-increasing}
\end{equation}

\subsection{Characterization via the STFT}
We now discuss the mapping properties of the STFT on the space $\mathcal{O}'_C(\mathcal{D},L^1_{\mathcal{W}})$. To do so, we need to introduce some more notation. We set $\mathcal{P} := ((1+| \: \cdot \: |)^{-n})_{n}$, a decreasing weight system,  and  write 
$$
\mathcal{P}C(\R^d) := \varinjlim_{n \in \N} C(1+| \: \cdot \: |)^{-n}(\R^d)
$$
for the corresponding $(LB)$-space of continuous functions. Furthermore, we denote by $\overline{P}$  the maximal Nachbin family associated to $\mathcal{P}$. More explicitly, the space $\overline{P}$ consists of all non-negative upper semicontinuous functions $p$ on $\R^d$ such that $$\sup_{\xi \in \R^d} p(\xi)(1+|\xi|)^{n} < \infty$$ for all $n \in \N$. For a non-negative function $f$ on $\R^d$ it holds that 
\begin{itemize}
\item[$(i)$]  $\sup_{\xi \in \R^d} f(\xi) (1+|\xi|)^{-n} < \infty$ for some $n \in \N$ if and only if $\sup_{\xi \in \R^d} f(\xi) p(\xi) < \infty$ for all $p \in \overline{P}$.
\item[$(ii)$]  $\sup_{\xi \in \R^d} f(\xi)/ p(\xi) < \infty$ for some $p \in \overline{P}$ if and only if  $\sup_{\xi \in \R^d} f(\xi)(1+|\xi|)^{n}  < \infty$ for all $n \in \N$.
\end{itemize}
We shall use these properties without explicitly referring to them. Let $\mathcal{W} = (w_N)_{N}$ be an increasing weight system. By \cite[Prop.\ 1.5]{Komatsu3} and \cite[Thm.\ 3.1(d)]{B-M-S}, we may identify $L^1_{\mathcal{W}}(\R_x^d) \widehat{\otimes}_\varepsilon \mathcal{P}C(\R_\xi^d)$ with the space consisting of all $f \in C(\R^d_\xi, L^1_{\mathcal{W}}(\R^d_x))$ 
satisfying the following property: For all $N  \in \N$ there is $n \in \N$ such that
$$
\sup_{\xi \in \R^d}\|f(\: \cdot \:, \xi)\|_{L^1_{w_N}} (1+|\xi|)^{-n}< \infty.
$$
Hence, a function $f \in C(\R^d_\xi, L^1_{\mathcal{W}}(\R^d_x))$ belongs to $L^1_{\mathcal{W}}(\R_x^d) \widehat{\otimes}_\varepsilon \mathcal{P}C(\R_\xi^d)$ if and only if 
$$
\| f \|_{L^1_{w_N},p} := \sup_{\xi \in \R^d}\|f(\: \cdot \:, \xi)\|_{L^1_{w_{N}}} p(\xi) < \infty
$$
for all $N \in \N$ and $p \in \overline{P}$. By \cite[Thm.\ 3.1(c)]{B-M-S}, the topology of $L^1_{\mathcal{W}}(\R_x^d) \widehat{\otimes}_\varepsilon\mathcal{P}C(\R_\xi^d)$ is generated by the system of seminorms $\{ \| \: \cdot \: \|_{L^1_{w_N},p} \, : \, N \in \N, p \in \overline{P} \}$. We then have:
\begin{proposition}\label{STFT-convolutors}
Let $\psi \in \mathcal{D}(\R^d)$ and let $\mathcal{W} = (w_N)_{N}$ be an increasing weight system satisfying \eqref{locally-bounded-increasing}. Then, the mappings
$$
V_\psi: \mathcal{O}'_C(\mathcal{D},L^1_{\mathcal{W}})\rightarrow  L^1_{\mathcal{W}}(\R_x^d) \widehat{\otimes}_\varepsilon \mathcal{P}C(\R_\xi^d)
$$ 
and
$$
V^\ast_\psi:  L^1_{\mathcal{W}}(\R_x^d) \widehat{\otimes}_\varepsilon \mathcal{P}C(\R_\xi^d) \rightarrow \mathcal{O}'_C(\mathcal{D},L^1_{\mathcal{W}})
$$
are well-defined and continuous. 
\end{proposition}
\begin{proof}
We first consider $V_\psi$. Observe that $V_\psi f( \: \cdot \:,\xi) \in L^1_{\mathcal{W}}(\R^d)$ for $\xi \in \R^d$ fixed, as follows from the representation $V_\psi f(x, \xi) = e^{-2\pi i x\xi} (f \ast M_\xi \check{\overline{\psi}})(x)$. We now prove that  $\R^d \rightarrow  L^1_{\mathcal{W}}(\R^d), \, \xi \mapsto V_\psi f(\: \cdot \: , \xi)$ is continuous. Since the mappings
$\R^d \rightarrow \mathcal{D}(\R^d), \, \xi \mapsto  M_\xi \check{\overline{\psi}}$ and $\mathcal{D}(\R^d) \rightarrow L^1_{\mathcal{W}}(\R^d), \, \varphi \mapsto f \ast \varphi$
are continuous, the mapping 
\begin{equation}
\R^d \rightarrow  L^1_{\mathcal{W}}(\R^d), \quad \xi \mapsto f \ast M_\xi \check{\overline{\psi}}
\label{continuity-conv-L1}
\end{equation}
 is also continuous. Let $\xi_0, \xi \in \R^d$ and $N \in \N$ be arbitrary. We have that
\begin{align*}
&\| V_\psi f(x, \xi) - V_\psi f(x , \xi_0) \|_{L^1_{w_N},x} \\
&=  \|  e^{-2\pi i x\xi} (f \ast M_\xi \check{\overline{\psi}})(x) -  e^{-2\pi i x \xi_0} (f \ast M_{\xi_0} \check{\overline{\psi}})(x)\|_{L^1_{w_N},x} \\
&\leq \| f \ast M_\xi \check{\overline{\psi}} - f \ast M_{\xi_0} \check{\overline{\psi}}\|_{L^1_{w_N}} + \|  (e^{-2\pi i x \xi} - e^{-2\pi i x \xi_0}) (f \ast M_{\xi_0} \check{\overline{\psi}})(x)\|_{L^1_{w_N},x}. 
\end{align*}
The first term tends to zero as $\xi \rightarrow \xi_0$ because the mapping \eqref{continuity-conv-L1} is continuous, while the second term tends to zero as $\xi \rightarrow \xi_0$ because of Lebesgue's dominated convergence theorem. Let $N \in \N$ and $p \in \overline{P}$ be arbitrary.  The set $B = \{  M_\xi \check{\overline{\psi}} p(\xi) \, : \, \xi \in \R^d\}$  is bounded in $\mathcal{D}(\R^d)$. Hence,
$$
\|V_\psi f\|_{L^1_{w_N},p} = \sup_{\xi \in \R^d}  \| f \ast M_\xi \check{\overline{\psi}} \|_{L^1_{w_N}} p(\xi) = \sup_{\varphi \in B} \| f \ast \varphi \|_{L^1_{w_N}},
$$
which shows that $V_\psi$ is well-defined and continuous. Next, we treat $V^\ast_\psi$. Let $B \subset \mathcal{D}(\R^d)$  bounded and $N \in \N$ be arbitrary. Choose $\widetilde{N} \geq N$ according to \eqref{locally-bounded-increasing}. Proposition \ref{STFT-D} implies that there are $K \Subset \R^d$ and $p \in \overline{P}$ such that  $\operatorname{supp} V_\psi \check{\overline{\varphi}} \subseteq K \times \R^d$ and $|V_\psi \check{\overline{\varphi}}(x,\xi)| \leq p(\xi)$ for all $(x,\xi) \in  K \times \R^d$ and $\varphi \in B$. Set $\widetilde{p} = p(\:\cdot\:)(1+ |\: \cdot \:|)^{d+1} \in \overline{P}$. Hence, 
\begin{align*}
\sup_{\varphi \in B} \| V^\ast_\psi F \ast \varphi \|_{L^1_{w_N}} &\leq \sup_{\varphi \in B} \int \int \int_{\R^{3d}} |F(x,\xi)| |V_\psi \check{\overline{\varphi}}(x-t,\xi)| w_N(t) \dx \dxi \dt \\
&\leq |K|\sup_{x \in K} \int \int_{\R^{2d}} |F(x +t,\xi)| p(\xi) w_N(t) \dxi \dt \\
& \leq C \|F\|_{\widetilde{p}, L^1_{w_{\widetilde{N}}}}
\end{align*}
for all $F \in L^1_{\mathcal{W}}(\R^d) \widehat{\otimes}_\varepsilon \mathcal{P}C(\R^d)$, where 
$
C = |K| \|\check{h}_{N,\widetilde{N}}\|_{K} \int_{\R^d} (1+|\xi|)^{-(d+1)} \dxi < \infty.
$
 This shows that $V^*_\psi$ is well-defined and continuous.
\end{proof}
Proposition \ref{STFT-convolutors} together with \eqref{reconstruction-D-dual} implies the following corollary:
\begin{corollary}\label{completeness-OC}
Let $\mathcal{W} = (w_N)_{N}$ be an increasing weight system satisfying \eqref{locally-bounded-increasing}. Then,  $\mathcal{O}'_C(\mathcal{D},L^1_{\mathcal{W}})$ is complete.
\end{corollary}
\subsection{A predual} Given an increasing weight system $\mathcal{W} =(w_N)_{N}$, we define its dual decreasing weight system as $\mathcal{W}^\circ = (1/w_N)_{N}$. The aim of this subsection is to show that the dual space $(\dot{\mathcal{B}}_{\mathcal{W}^\circ}(\R^d))'$, endowed with a suitable $\mathfrak{S}$-topology, is topologically equal to $\mathcal{O}'_C(\mathcal{D},L^1_{\mathcal{W}})$. We start by introducing the following general notion:
\begin{definition}
Let $E = \varinjlim_{N} E_N$ be an $(LF)$-space. We define 
$$
\mathfrak{S} = \{ B \subset E \, : \, B \mbox{ is contained and bounded in $E_N$ for some $N \in \N$} \}
$$
and write $bs(E',E)$ for the $\mathfrak{S}$-topology on $E'$ (the topology of uniform convergence on the sets of $\mathfrak{S}$). Grothendieck's factorization theorem implies that $bs(E',E)$ does not depend on the defining inductive spectrum of $E$.
\end{definition}
\begin{remark}
Let $E = \varinjlim_{N} E_N$ be an $(LF)$-space. The bipolar theorem implies that $bs(E', E) = b(E',E)$ if and only if $E$ is \emph{quasiregular}, i.e.,  if for every bounded set $B$ in $E$ there is $N \in \N$ and a bounded set $A$ in  $E_N$ such that $B \subseteq \overline{A}^{E}$. We refer to \cite{D-D} for more information on quasiregular $(LF)$-spaces.
\end{remark}

We then have:
\begin{theorem}\label{thm-predual}
Let $\mathcal{W} = (w_N)_{N}$ be an increasing weight system satisfying \eqref{locally-bounded-increasing}. Then, $(\dot{\mathcal{B}}_{\mathcal{W}^\circ}(\R^d))'_{bs}= \mathcal{O}'_C(\mathcal{D}, L^1_\mathcal{W})$.
\end{theorem}
We need some preparation. The proof of the next proposition is standard and therefore omitted.
\begin{proposition}
\label{convspaces-proposition-structural}
Let $w$ be a positive  continuous function and let $f \in (\dot{\mathcal{B}}_{1/w}(\R^d))'$. Then, there are $n \in \N$ and regular complex Borel measures $\mu_\alpha \in (C_0(\R^d))'$, $|\alpha| \leq n$, such that
\begin{equation}
\langle f, \varphi \rangle = \sum_{|\alpha| \leq n} (-1)^{|\alpha|} \int_{\R^d} \frac{\partial^\alpha \varphi(t)}{w(t)} {\rm d}\mu_\alpha(t), \qquad \varphi \in  \dot{\mathcal{B}}_{1/w}(\R^d).
\label{repr-dual-bdot}
\end{equation}
\end{proposition}
\begin{corollary}\label{conv-b-dot}  Let $w$ and $v$ be positive continuous functions on $\R^d$ satisfying \eqref{locally-bounded-cond} and let $f \in (\dot{\mathcal{B}}_{1/w}(\R^d))'$. Then,
\begin{itemize}
\item[$(i)$] $f\ast \varphi \in L^1_v(\R^d)$ for all $\varphi \in \mathcal{D}(\R^d)$.
\item[$(ii)$] For all $\varphi \in \mathcal{D}(\R^d)$ and $h \in C(1/v)_0(\R^d)$ it holds that $\check{\varphi} \ast h \in \dot{\mathcal{B}}_{1/w}(\R^d)$ and
$$
\int_{\R^d}( f \ast \varphi) (x) h(x) \dx = \langle f, \check{\varphi} \ast h \rangle.
$$
\end{itemize}
\end{corollary}
\begin{proof}
Assume that $f$ is represented via \eqref{repr-dual-bdot}. In particular,
$$
f\ast \varphi = \sum_{|\alpha| \leq n} \int_{\R^d} \frac{\partial^\alpha\varphi(\: \cdot \: - t)}{w(t)} {\rm d}\mu_\alpha(t), \qquad \varphi \in  \mathcal{D}(\R^d).
$$

$(i)$ Let $\varphi \in \mathcal{D}(\R^d)$ be arbitrary and set $K = \operatorname{supp} \varphi$.  We have that 
$$
\| f\ast \varphi \|_{L^1_v} \leq  \sum_{|\alpha| \leq n} \int_{\R^d} \int_{\R^d} \frac{|\partial^\alpha\varphi(x - t)|}{w(t)} {\rm d}|\mu_\alpha|(t) v(x)\dx,
$$
where $|\mu_\alpha|$ denotes the total variation measure associated with $\mu_\alpha$. For each $|\alpha| \leq n$ it holds that
\begin{align*}
\int_{\R^d} \int_{\R^d} |\partial^\alpha\varphi(x - t)| v(x) \dx \frac{1}{w(t)} {\rm d} |\mu_\alpha|(t) &=  \int_{\R^d} \int_{\R^d} |\partial^\alpha\varphi(x)| v(x + t) \dx \frac{1}{w(t)} {\rm d} |\mu_\alpha|(t) \\
& \leq |K| \|\varphi\|_{K,n} \|g\|_K |\mu_\alpha|(\R^d) < \infty.
\end{align*}
Hence, Fubini's theorem implies that $\| f\ast \varphi \|_{L^1_v} < \infty$.

$(ii)$  Let $\varphi \in \mathcal{D}(\R^d)$ and $h \in C(1/v)_0(\R^d)$ be arbitrary. Set $K = \operatorname{supp} \varphi$. For each $\alpha \in \N^d$ we have that
\begin{equation}
\frac{|\partial^{\alpha}(\check{\varphi} \ast h)(x)|}{w(x)} \leq |K| \|\partial^\alpha \varphi\|_K \|g\|_K \sup_{t \in K} \frac{|h(t+x)|}{v(t+x)},
\label{bounded-set-later}
\end{equation}
which tends to zero as $|x| \to \infty$, that is, $\check{\varphi} \ast h \in \dot{\mathcal{B}}_{1/w}(\R^d)$. Finally,
\begin{align*}
\int_{\R^d} (f \ast \varphi) (x) h(x) \dx &=  \sum_{|\alpha| \leq n} \int_{\R^d} \int_{\R^d} \frac{\partial^\alpha\varphi(x - t)}{w(t)} {\rm d}\mu_\alpha(t)h(x) \dx \\
& =  \sum_{|\alpha| \leq n} \int_{\R^d} \int_{\R^d} \partial^\alpha\varphi(x - t)h(x) \dx  \frac{1}{w(t)} {\rm d}\mu_\alpha(t)  \\
& =  \sum_{|\alpha| \leq n} (-1)^{|\alpha|} \int_{\R^d}  \frac{\partial^\alpha(\check{\varphi} \ast h)(t)}{w(t)} {\rm d}\mu_\alpha(t)  \\
&=  \langle f, \check{\varphi} \ast h \rangle.
\qedhere
\end{align*}
\end{proof}
\begin{proof}[Proof of Theorem \ref{thm-predual}]
From Corollary \ref{conv-b-dot}$(i)$, we obtain that $(\dot{\mathcal{B}}_{\mathcal{W}^\circ}(\R^d))' \subseteq \mathcal{O}'_C(\mathcal{D}, L^1_\mathcal{W})$. We now show that this inclusion holds continuously if we endow the former space with its $bs$-topology. For this, we need the following result from measure theory (cf.\ \cite[Thm.\ 6.9 and Thm.\ 6.13]{Rudin}): Let $w$ be a positive continuous function on $\R^d$ and let $f \in L^1_w(\R^d)$. Denote by $B$ the unit ball in $C_0(\R^d)$. Then,
$$
\|f\|_{L^1_w} = \sup_{h \in B} \left | \int_{\R^d} f(x)h(x) w(x)\dx \right|.
$$ 

Let $B' \subset \mathcal{D}(\R^d)$ bounded and $N \in \N$ be arbitrary. By the above remark, we have that
$$
\sup_{\varphi \in B'} \| f \ast \varphi\|_{L^1_{w_N}} = \sup_{\varphi \in B'} \sup_{h \in B} \left | \int_{\R^d} f\ast \varphi(x) h(x)w_N(x) \dx \right | = \sup_{\varphi \in B'} \sup_{h \in B} \left | \langle f, \check{\varphi} \ast (hw_N) \rangle \right|, 
$$
for all $f \in (\dot{\mathcal{B}}_{\mathcal{W}^\circ}(\R^d))'$, where the last equality follows from Corollary \ref{conv-b-dot}$(ii)$. Choose $\widetilde{N} \geq N$ according to $\eqref{locally-bounded-increasing}$. It is now enough to notice that the set 
$$
\{ \check{\varphi} \ast (hw_N) \, : \, \varphi \in B', h \in B \}
$$
is bounded in $\dot{\mathcal{B}}_{1/w_{\widetilde{N}}}(\R^d)$, as follows from \eqref{bounded-set-later} with $v = w_N $ and $w =w_{\widetilde{N}}$. Next, we show that $\mathcal{O}'_C(\mathcal{D}, L^1_\mathcal{W})$ is continuously included in $(\dot{\mathcal{B}}_{\mathcal{W}^\circ}(\R^d))'_{bs}$. Let $\gamma \in \mathcal{D}(\R^d)$. By \eqref{reconstruction-D-dual} and Proposition \ref{STFT-convolutors}, it suffices to show that the mapping
 $$
V^\ast_\gamma:  L^1_{\mathcal{W}}(\R_x^d) \widehat{\otimes}_\varepsilon \mathcal{P}C(\R_\xi^d) \rightarrow (\dot{\mathcal{B}}_{\mathcal{W}^\circ}(\R^d))'_{bs}
$$
is  well-defined and continuous.  Let $F \in L^1_{\mathcal{W}}(\R_x^d) \widehat{\otimes}_\varepsilon \mathcal{P}C(\R_\xi^d)$ be arbitrary. The linear functional 
$$
f: \dot{\mathcal{B}}_{\mathcal{W}^\circ}(\R^d) \rightarrow \C, \quad \varphi \mapsto \int \int_{\R^{2d}} F(x,\xi) V_{\overline{\gamma}} \varphi(x,-\xi)\dx\dxi
$$
is well-defined and continuous by Lemma \ref{STFT-test-smooth}. Since $V^\ast_\gamma F = f_{| \mathcal{D}(\R^d)}$, we obtain that $V^\ast_\gamma F \in  (\dot{\mathcal{B}}_{\mathcal{W}^\circ}(\R^d))'$ and
$$
\langle V_\gamma^\ast F, \varphi \rangle =  \int \int_{\R^{2d}} F(x,\xi) V_{\overline{\gamma}} \varphi(x,-\xi)\dx\dxi, \qquad \varphi \in \dot{\mathcal{B}}_{\mathcal{W}^\circ}(\R^d),
$$
whence $V^\ast_\gamma$ is well-defined. Finally, we show that it is continuous. Let $N \in \N$ and $B \subset \dot{\mathcal{B}}_{1/w_N}(\R^d)$ bounded be arbitrary. Choose $\widetilde{N} \geq N$ according to \eqref{locally-bounded-increasing}. Lemma \ref{STFT-test-smooth} implies that there is $p \in \overline{P}$ such that $|V_{\overline{\gamma}} \varphi(x,-\xi)| \leq w_{\widetilde{N}}(x) p(\xi)$ for all $(x,\xi) \in \R^{2d}$ and $\varphi \in B$. Set $\widetilde{p} = p(\: \cdot \:) (1+ |\: \cdot \:|)^{d+1} \in \overline{P}$. Hence,
\begin{align*}
\sup_{\varphi \in B} |\langle V^\ast_\gamma F, \varphi \rangle | & \leq \sup_{\varphi \in B} \int \int_{\R^{2d}} |F(x,\xi)| |V_{\overline{\gamma}} \varphi(x,-\xi)| \dx \dxi\\
&\leq \int \int_{\R^{2d}} |F(x,\xi)|w_{\widetilde{N}}(x) p(\xi) \dx \dxi \leq  \|F\|_{\widetilde{p}, L^1_{w_{\widetilde{N}}}}\int_{\R^d} (1+|\xi|)^{-(d+1)}\dxi
\end{align*}
for all $F \in L^1_{\mathcal{W}}(\R_x^d) \widehat{\otimes}_\varepsilon \mathcal{P}C(\R_\xi^d)$.
\end{proof}
From now on, we shall interchangeably use  $\mathcal{O}'_C(\mathcal{D}, L^1_{\mathcal{W}})$ and $(\dot{\mathcal{B}}_{\mathcal{W}^\circ}(\R^d))'_{bs}$ depending on which point of view is most suitable for the given situation. We shall not explicitly refer to Theorem \ref{thm-predual} when we do this.  For later use, we point out the following corollary.
\begin{corollary}\label{desing-dual-b-dot}
Let $\psi \in \mathcal{D}(\R^d) \backslash \{0\}$ and $\gamma \in \mathcal{D}(\R^d)$ be a synthesis window for $\psi$. Let $\mathcal{W} = (w_N)_{N}$ be an increasing weight system satisfying \eqref{locally-bounded-increasing}. Then, the desingularization formula 
$$
\langle f, \varphi \rangle = \frac{1}{(\gamma,\psi)_{L^2}} \int \int_{\R^{2d}} V_\psi f(x,\xi) V_{\overline{\gamma}} \varphi(x,-\xi) \dx \dxi
$$
holds for all $f \in (\dot{\mathcal{B}}_{\mathcal{W}^\circ}(\R^d))'$ and $\varphi \in \dot{\mathcal{B}}_{\mathcal{W}^\circ}$. 
\end{corollary}
\subsection{Topological properties} We now take a closer look at the locally convex structure of
$\mathcal{O}'_C(\mathcal{D},L^1_{\mathcal{W}})$. Most importantly, we are going to show that $\mathcal{O}'_C(\mathcal{D},L^1_{\mathcal{W}})$ is ultrabornological if and only if $\mathcal{W}^\circ$ satisfies $(\Omega)$. Our main tools will be Theorem \ref{completeness-ind-lim-smooth} and an explicit description of the strong dual of $\mathcal{O}'_C(\mathcal{D},L^1_{\mathcal{W}})$.   The following technical lemma is needed below.
\begin{lemma}\label{density-smooth}
Let $\mathcal{W} = (w_N)_{N}$ be an increasing weight system satisfying \eqref{locally-bounded-increasing}. Then, we have the dense continuous inclusion $\mathcal{D}(\R^d) \hookrightarrow  \mathcal{O}'_C(\mathcal{D},L^1_{\mathcal{W}})$. 
\end{lemma}
\begin{proof} Notice that $\mathcal{D}(\R^d) \subset \mathcal{D}_{L^1_{\mathcal{W}}}(\R^d) \subset\mathcal{O}'_C(\mathcal{D},L^1_{\mathcal{W}})$ with continuous inclusion mappings. We shall prove that both inclusion mappings have dense range.  We start by showing that $\mathcal{D}_{L^1_{\mathcal{W}}}(\R^d)$ is dense in $\mathcal{O}'_C(\mathcal{D},L^1_{\mathcal{W}})$. Let $f \in \mathcal{O}'_C(\mathcal{D},L^1_{\mathcal{W}})$ be arbitrary. Choose $\chi \in \mathcal{D}(\R^d)$ with $\int_{\R^d} \chi(x) \dx =1$ and $\operatorname{supp} \chi \subseteq \overline{B}(0,1)$. Set $\chi_k = k^d \chi(k \: \cdot \:)$ and $f_k = f \ast \chi_k \in \mathcal{D}_{L^1_{\mathcal{W}}}(\R^d)$ for $k \in \Z_+$ (cf.\ Lemma \ref{new}). We claim that $f_k \rightarrow f$ in $\mathcal{O}'_C(\mathcal{D},L^1_{\mathcal{W}})$. Let $B \subset \mathcal{D}(\R^d)$ bounded and $N \in \N$  be arbitrary. Choose $\widetilde{N} \geq N$ according to \eqref{locally-bounded-increasing}. 
Hence,
\begin{align*}
&\sup_{\varphi \in B} \| f_k \ast \varphi - f \ast \varphi \|_{L^1_{w_N}} \\
&\leq \sup_{\varphi \in B} \| f \ast \varphi \ast \chi_k - f \ast \varphi \|_{L^1_{w_N}} \\
&\leq \sup_{\varphi \in B} \int_{\R^d} \int_{\R^d} |\chi(t)| | (f\ast \varphi) (x-t/k) - (f \ast \varphi)(x)| \dt w_N(x) \dx \\
&\leq \frac{1}{k}\sup_{\varphi \in B} \int_{\R^d} \int_{\R^d} |\chi(t)| |t| \sum_{|\beta|=1} \int_{0}^1 |\partial^{\beta}(f\ast \varphi)(x-\gamma t/k)| \mathrm{d}\gamma \dt w_N(x) \dx \\
&\leq \frac{1}{k}  \sum_{|\beta|=1} \sup_{\varphi \in B} \int_0^1 \int_{\overline{B}(0,1)} |\chi(t)||t| h_{N, \widetilde{N}}(\gamma t/k) \times \\
&\phantom{\leq} \int_{\R^d}  |(f\ast \partial^{\beta}\varphi)(x-\gamma t/k)| 
w_{\widetilde{N}}(x-\gamma t/k) \dx \dt \mathrm{d}\gamma \\
&\leq \frac{C}{k},
\end{align*}
where 
$
C = \| \chi(t)t\|_{L_1,t} \|h_{N, \widetilde{N}}\|_{\overline{B}(0,1)}\sum_{|\beta| = 1} \sup_{\varphi \in B} \| f \ast \partial^\beta \varphi \|_{L^1_{w_{\widetilde{N}}}}  < \infty.
$ 
Next, we show that $\mathcal{D}(\R^d)$ is dense in $\mathcal{D}_{L^1_{\mathcal{W}}}(\R^d)$. Let $\varphi \in \mathcal{D}_{L^1_{\mathcal{W}}}(\R^d)$ be arbitrary. Choose $\theta \in \mathcal{D}(\R^d)$ with $0 \leq \theta \leq 1$ and $\theta(0) = 1$. Set $\theta_k = \theta(\: \cdot \: /k)$ and $\varphi_k = \theta_k\varphi \in \mathcal{D}(\R^d)$ for $k \in \Z_+$.  We claim that $\varphi_k \rightarrow \varphi$ in $\mathcal{D}_{L^1_{\mathcal{W}}}(\R^d)$. Let $N,n \in \N$ be arbitrary. Hence,
\begin{align*}
\max_{|\alpha| \leq n} \| \partial^\alpha(\varphi-\varphi_k)\|_{L^1_{w_N}} &\leq \max_{|\alpha| \leq n} \| \partial^\alpha\varphi(1-\theta_k)\|_{L^1_{w_N}} \\
&+ \frac{1}{k} \max_{|\alpha| \leq n} \sum_{\beta \leq \alpha, \beta \neq 0} \binom{\alpha}{\beta} \| \partial^{\alpha-\beta}\varphi \partial^{\beta}\theta(\: \cdot \: /k)\|_{L^1_{w_N}}\\
&\leq \max_{|\alpha| \leq n} \| \partial^\alpha\varphi(1-\theta_k)\|_{L^1_{w_N}} + \frac{C}{k},
\end{align*}
where $C = 2^n\max_{|\alpha| \leq n}\|\partial^\alpha\varphi\|_{L^1_{w_N}}\|\theta\|_{\operatorname{supp}\theta,n} < \infty.$ We still need to show that  $\| \partial^\alpha\varphi(1-\theta_k)\|_{L^1_{w_N}} \rightarrow 0$ for all $|\alpha| \leq n$.  Let $\varepsilon > 0$ be arbitrary. Choose $K \Subset \R^d$ so large that
$$\int_{\R^d \backslash K} |\partial^\alpha \varphi (x)| w_N(x) \dx \leq \varepsilon.$$
Then,
\begin{align*}
&\| \partial^\alpha\varphi(1-\theta_k)\|_{L^1_{w_N}}  \\
&= \int_{K} |\partial^\alpha \varphi (x)|(1-\theta_k(x)) w_N(x) \dx  + \int_{\R^d \backslash K} |\partial^\alpha \varphi (x)| (1-\theta_k(x)) w_N(x) \dx  \\
&\leq C' \sup_{x \in K} |1-\theta_k(x)| + \varepsilon,
\end{align*}
where $C' = |K| \|\partial^\alpha \varphi\|_K \|w_N\|_K < \infty.$ The result  follows from the fact that, since $\theta(0) = 1$,  $\theta_k \rightarrow 1$ uniformly on $K$.
\end{proof}
\begin{proposition}\label{thm-dual-convolutors}
Let $\mathcal{W} = (w_N)_{N}$ be an increasing weight system satisfying \eqref{locally-bounded-increasing}.
\begin{itemize}
\item[$(i)$]  The canonical inclusion 
\begin{equation}
\dot{\mathcal{B}}_{\mathcal{W}^\circ}(\R^d) \rightarrow ((\dot{\mathcal{B}}_{\mathcal{W}^\circ}(\R^d))'_{bs})'_b, \quad \varphi \mapsto (f \mapsto \langle f, \varphi \rangle)
\label{canonical-incl}
\end{equation}
is a topological embedding.
\item[$(ii)$]  Let $\psi \in \mathcal{D}(\R^d) \backslash \{0\}$ and $\gamma \in \mathcal{D}(\R^d)$ be a synthesis window for $\psi$. Then, the mapping $\mathcal{B}_{\mathcal{W}^\circ}(\R^d) \rightarrow  ((\dot{\mathcal{B}}_{\mathcal{W}^\circ}(\R^d))'_{bs})'_{b}$ given by

\begin{equation}
\varphi \mapsto \left( f \mapsto \frac{1}{(\gamma,\psi)_{L^2}}\int \int_{\R^{2d}} V_\psi f(x,\xi) V_{\overline{\gamma}}\varphi(x,-\xi) \dx \dxi \right)
\label{dual-convolutors}
\end{equation}
is a continuous bijection whose restriction to $\dot{\mathcal{B}}_{\mathcal{W}^\circ}(\R^d)$ coincides with the canonical inclusion \eqref{canonical-incl}.
\end{itemize}
\end{proposition}
\begin{proof}
$(i)$ The $bs$-topology is coarser than the strong topology and finer than the weak-$\ast$ topology on $(\dot{\mathcal{B}}_{\mathcal{W}^\circ}(\R^d))'$. Since $\dot{\mathcal{B}}_{\mathcal{W}^\circ}(\R^d)$ is barrelled (as an $(LF)$-space), a subset of $(\dot{\mathcal{B}}_{\mathcal{W}^\circ}(\R^d))'$ is equicontinuous if and only if it is bounded in the $bs$-topology, which in turn yields that \eqref{canonical-incl} is a strict morphism.

$(ii)$ We start by showing that, for $\varphi \in \mathcal{B}_{\mathcal{W}^\circ}(\R^d)$ fixed, the linear functional
$$
\mathcal{O}'_C(\mathcal{D},L^1_\mathcal{W}) \rightarrow \C, \quad  f \mapsto \frac{1}{(\gamma,\psi)_{L^2}}\int \int_{\R^{2d}} V_\psi f(x,\xi) V_{\overline{\gamma}}\varphi(x,-\xi) \dx \dxi 
$$
is well-defined and continuous. Proposition \ref{STFT-test-char-smooth} implies that there are $N \in \N$ and $p \in \overline{P}$ such that $|V_{\overline{\gamma}}\varphi(x,-\xi)| \leq w_N(x) p(\xi)$ for all $(x,\xi) \in \R^{2d}$. Set $\widetilde{p} = p(\: \cdot \:) (1+ |\: \cdot \:|)^{d+1} \in \overline{P}$. Then,
\begin{align*}
&\left | \frac{1}{(\gamma,\psi)_{L^2}}\int \int_{\R^{2d}} V_\psi f(x,\xi) V_{\overline{\gamma}}\varphi(x,-\xi) \dx \dxi \right| \\
&\leq \frac{1}{|(\gamma,\psi)_{L^2}|}\int \int_{\R^{2d}} |V_\psi f(x,\xi)| w_N(x) p(\xi) \dx \dxi\\
&\leq \frac{\| V_\psi f \|_{L^1_{w_N, \widetilde{p}} }}{|(\gamma,\psi)_{L^2}|} \int_{\R^d} (1+|\xi|)^{-(d+1)} \dx
\end{align*}
 for all $f \in \mathcal{O}'_C(\mathcal{D},L^1_\mathcal{W})$.
The claim  follows from Proposition \ref{STFT-convolutors}. Next, we prove that \eqref{dual-convolutors} is continuous. It suffices to show that its restriction to $\mathcal{B}_{1/w_N}(\R^d)$ is continuous for each $N \in \N$. Let $B \subset \mathcal{O}'_C(\mathcal{D},L^1_\mathcal{W})$ bounded be arbitrary. Choose $\widetilde{N} \geq N$ according to \eqref{locally-bounded-increasing}. 
By Proposition \ref{STFT-convolutors}, there are $n \in \N$  and $C > 0$ such that
$$
\|V_\psi f(\: \cdot \:, \xi)\|_{L^1_{w_{\widetilde{N}}}} \leq C(1+|\xi|)^{n}, \qquad \xi \in \R^d,
$$
for all $f \in B$. Lemma \ref{STFT-test-smooth} implies that there is $C' > 0$ such that 
$$
|V_{\overline{\gamma}}\varphi(x,-\xi)| \leq  \frac{C'\|\varphi\|_{\mathcal{B}^{n+d+1}_{1/w_N}}w_{\widetilde{N}}(x)}{(1+|\xi|)^{n+d+1}}, \qquad (x,\xi) \in \R^{2d},
 $$
 for all $\varphi \in \mathcal{B}_{1/w_N}(\R^d)$. Hence,
$$
\sup_{f \in B} \left | \frac{1}{(\gamma,\psi)_{L^2}}\int \int_{\R^{2d}} V_\psi f(x,\xi) V_{\overline{\gamma}}\varphi(x,-\xi) \dx \dxi \right |  \leq C''\|\varphi\|_{\mathcal{B}^{n+d+1}_{1/w_N}},
$$
for all $\varphi \in \mathcal{B}_{1/w_N}(\R^d)$, where 
$
C''= \frac{CC'}{|(\gamma,\psi)_{L^2}|} \int_{\R^d} (1+|\xi|)^{-(d+1)} < \infty.
$
Finally, we prove that \eqref{dual-convolutors} is surjective. Let $\Phi \in (\mathcal{O}'_C(\mathcal{D},L^1_{\mathcal{W}}))'$ be arbitrary. Denote by $\iota: \mathcal{D}(\R^d) \rightarrow \mathcal{O}'_C(\mathcal{D},L^1_{\mathcal{W}})$ the inclusion mapping and set $f = \Phi \circ \iota \in \mathcal{D}'(\R^d)$. By \eqref{desing-D-dual}, it holds that
$$
\Phi(\iota(\chi)) = \langle f, \chi \rangle = \frac{1}{(\gamma,\psi)_{L^2}}\int \int_{\R^{2d}} V_\psi \chi(x,\xi) V_{\overline{\gamma}}f(x,-\xi) \dx\dxi
$$
for all $\chi \in \mathcal{D}(\R^d)$. By Lemma \ref{density-smooth}, it therefore suffices to show that $f \in \mathcal{B}_{\mathcal{W}^\circ}(\R^d)$ or, equivalently, that $V_\theta f \in \mathcal{W}^\circ_{\operatorname{pol}}C(\R^{2d})$, where $\theta \in \mathcal{D}(\R^d)$ is a fixed non-zero window function (Proposition \ref{STFT-test-char-smooth}). Since $\Phi$ is continuous, there is $N \in \N$ and a bounded set $B \subset \dot{\mathcal{B}}_{1/w_N}(\R^d)$ such that 
\begin{align*}
|V_\theta f(x,\xi)| =  |\Phi(\iota( \overline{M_\xi T_x\theta}))| \leq  \sup_{\varphi \in B}\left| \int_{\R^d} \varphi(t)  \overline{M_\xi T_x\theta(t)} \dt \right | = \sup_{\varphi \in B} |V_\theta \varphi(x,\xi)|.
\end{align*}
The required bounds for $|V_\theta f|$ therefore directly follow from Lemma \ref{STFT-test-smooth}. The last statement is a reformulation of Corollary \ref{desing-dual-b-dot}. 
\end{proof}
In view of Corollary \ref{proj-desc-smooth-functions} and the fact that $\mathcal{D}(\R^d)$ is bornological, a slight modification of the argument used  in the last part of the proof of Proposition \ref{thm-dual-convolutors}$(ii)$ yields the following corollary.
\begin{corollary}\label{last-cor-ever}
Let $\mathcal{W} = (w_N)_{N}$ be an increasing weight system satisfying \eqref{locally-bounded-increasing} such that $\mathcal{W}^\circ$ satisfies $(\Omega)$. Then, every bounded linear functional on $\mathcal{O}'_C(\mathcal{D}, L^1_\mathcal{W})$ is continuous.
\end{corollary}

Next, we point out two situations in which the mapping \eqref{dual-convolutors} is a topological isomorphism.
\begin{corollary}
Let $\mathcal{W} = (w_N)_{N}$ be an increasing weight system satisfying \eqref{locally-bounded-increasing} such that $\mathcal{W}^\circ$ satisfies $(V)$  (cf.\ Remark \ref{V-remark}). Then, the mapping
$$
\mathcal{B}_{\mathcal{W}^\circ}(\R^d) \rightarrow ((\mathcal{B}_{\mathcal{W}^\circ}(\R^d))'_{bs})'_{b}, \quad \varphi \mapsto (f \mapsto \langle f, \varphi \rangle)
$$
is a topological isomorphism.
\end{corollary}
\begin{corollary}\label{bidual-Omega}
Let $\mathcal{W} = (w_N)_{N}$ be an increasing weight system satisfying \eqref{locally-bounded-increasing} such that $\mathcal{W}^\circ$ satisfies $(\Omega)$. Then, the mapping \eqref{dual-convolutors} is a topological isomorphism.
\end{corollary}
\begin{proof}
We need to show that the mapping \eqref{dual-convolutors} is open. Fix a non-zero window function $\theta \in \mathcal{D}(\R^d)$. By Corollaries  \ref{desing-E-dual} and \ref{proj-desc-smooth-functions}, it suffices to show that for every $u \in \overline{U}( \mathcal{W}^\circ_{\operatorname{pol}})$ there is a set $B \subset \mathcal{E}'(\R^d)$ which is bounded in $(\dot{\mathcal{B}}_{\mathcal{W}^\circ}(\R^d))'_{bs}$ such that
$$
\|V_\theta \varphi \|_{u} \leq \sup_{f \in B} | \langle f, \varphi \rangle|
$$
for all $\varphi \in \mathcal{B}_{\mathcal{W}^\circ}(\R^d)$. Define 
$$
B = \{ u(x,\xi) \overline{M_\xi T_x \theta} \, : \, (x,\xi) \in \R^{2d}\} \subset \mathcal{E}'(\R^d)
$$
and observe that $B$ satisfies all requirements because
$$
\|V_\theta \varphi \|_{u} = \sup_{(x,\xi) \in \R^{2d}} |V_\theta \varphi(x,\xi)|u(x,\xi) = \sup_{(x,\xi) \in \R^{2d}} |\langle \overline{M_\xi T_x \theta}, \varphi \rangle| u(x,\xi) = \sup_{f \in B} |\langle f, \varphi \rangle | 
$$
for all $\varphi \in \mathcal{B}_{\mathcal{W}^\circ}(\R^d)$.
\end{proof}
We believe that the mapping \eqref{dual-convolutors} might  \emph{always} be a topological isomorphism but were  unable to prove this in general. 

We are  ready to  prove the main theorem of this section.  A lcHs is said to be quasibarrelled (sometimes also called infrabarrelled)  if every strongly bounded set in its dual is equicontinuous. Recall that a lcHs $E$ is bornological (ultrabornological, respectively) if and only if it is quasibarrelled and every linear functional on $E$ that is bounded on the bounded subsets of $E$ (on the Banach disks of $E$, respectively), is continuous. 
In particular, every complete bornological lcHs  is  ultrabornological. 

\begin{theorem}\label{char-UB-smooth}
Let $\mathcal{W} = (w_N)_{N}$ be an increasing weight system satisfying \eqref{locally-bounded-increasing}. Then, the following statements are equivalent:
\begin{itemize}
\item[$(i)$] $\mathcal{W}^\circ$ satisfies $(\Omega)$.
\item[$(ii)$] $\mathcal{O}'_C(\mathcal{D}, L^1_\mathcal{W})$ is ultrabornological.
\item[$(iii)$] $(\dot{\mathcal{B}}_{\mathcal{W}^\circ}(\R^d))'_{b}= \mathcal{O}'_C(\mathcal{D}, L^1_\mathcal{W})$.
\end{itemize}
\end{theorem}
\begin{proof}

$(i) \Rightarrow (ii)$  Since $\mathcal{O}'_C(\mathcal{D}, L^1_\mathcal{W})$ is complete (Corollary \ref{completeness-OC}), we only need to prove that it is bornological. We already pointed out that every bounded linear functional on $\mathcal{O}'_C(\mathcal{D}, L^1_\mathcal{W})$ is continuous (Corollary \ref{last-cor-ever}). We now show that $\mathcal{O}'_C(\mathcal{D}, L^1_\mathcal{W})$ is quasibarrelled. 
Choose $\psi, \gamma \in \mathcal{D}(\R^d)$ with $(\gamma, \psi)_{L^2} = 1$.  By Corollary \ref{bidual-Omega} and Proposition \ref{STFT-convolutors}, it suffices to show that for every bounded set $B \subset \mathcal{B}_{\mathcal{W}^\circ}(\R^d)$ there are $\widetilde{p} \in \overline{P}$, $\widetilde{N} \in \N$, and $C >0$ such that 
$$
\sup_{\varphi \in B} \left |\int \int_{\R^{2d}} V_\psi f(x,\xi) V_{\overline{\gamma}}\varphi(x,-\xi) \dx\dxi  \right | \leq C\| V_\psi f\|_{\widetilde{p}, L^1_{w_{\widetilde{N}}}}
$$
for all $f \in \mathcal{O}'_C(\mathcal{D}, L^1_\mathcal{W})$. Since $\mathcal{B}_{\mathcal{W}^\circ}(\R^d)$ is regular (Theorems \ref{reg-cond} and \ref{completeness-ind-lim-smooth}), there is $N \in \N$ such that $B$ is contained and bounded in $\mathcal{B}_{1/w_N}(\R^d)$. Choose $\widetilde{N} \geq N$ according to  \eqref{locally-bounded-increasing}. Lemma \ref{STFT-test-smooth} implies that there is $p \in \overline{P}$ such that $|V_{\overline{\gamma}}\varphi(x,-\xi)| \leq w_{\widetilde{N}}(x)p(\xi)$ for all $(x,\xi) \in \R^{2d}$ and $\varphi \in B$.  Set $\widetilde{p} = p(\: \cdot \:) (1+ |\: \cdot \:|)^{d+1} \in \overline{P}$. We obtain that
$$
\sup_{\varphi \in B} \left |\int \int_{\R^{2d}} V_\psi f(x,\xi) V_{\overline{\gamma}}\varphi(x,-\xi) \dx\dxi \right | \leq \|V_\psi f\|_{\widetilde{p},L^1_{w_{\widetilde{N}}}} \int_{\R^d}(1+|\xi|)^{-(d+1)}\dxi 
$$
for all $f \in \mathcal{O}'_C(\mathcal{D}, L^1_\mathcal{W})$.

$(ii) \Rightarrow (iii)$ It suffices to show that the $bs$-topology is finer than the strong topology on $(\dot{\mathcal{B}}_{\mathcal{W}^\circ}(\R^d))'$. Let $U \subset  (\dot{\mathcal{B}}_{\mathcal{W}^\circ}(\R^d))'_b$ be an arbitrary neighborhood of $0$. We may assume that $U = B^\circ$, where $B^\circ$ is the polar set of some bounded set $B$ in $\dot{\mathcal{B}}_{\mathcal{W}^\circ}(\R^d)$.  Since \eqref{canonical-incl} is continuous and $(\dot{\mathcal{B}}_{\mathcal{W}^\circ}(\R^d))'_{bs}$ is ultrabornological, there is $N \in \N$ and a bounded set $B' \subset  \dot{\mathcal{B}}_{1/w_N}(\R^d)$ such that $(B')^\circ \subseteq U$.

$(iii) \Rightarrow (i)$  By Theorems \ref{reg-cond} and \ref{completeness-ind-lim-smooth}, it suffices to show that $\dot{\mathcal{B}}_{\mathcal{W}^\circ}(\R^d)$ is $\beta$-regular. Let $N \in \N$ be arbitrary and let $B$ be a subset of $\dot{\mathcal{B}}_{1/w_N}(\R^d)$ that is bounded in $\dot{\mathcal{B}}_{\mathcal{W}^\circ}(\R^d)$. Our assumption implies that there is $M \geq N$ and a bounded set $B' \subset \dot{\mathcal{B}}_{1/w_M}(\R^d)$ such that $(B')^\circ \subseteq B^\circ$, where the polarity is taken twice with respect to the dual system $(\dot{\mathcal{B}}_{\mathcal{W}^\circ}(\R^d),(\dot{\mathcal{B}}_{\mathcal{W}^\circ}(\R^d))')$. We now show that $B$ is bounded in $\dot{\mathcal{B}}_{1/w_M}(\R^d)$. Let $n \in \N$ be arbitrary and choose $C > 0$ such that $\| \varphi \|_{1/w_M,n} \leq C$ for all $\varphi \in B'$. Hence,
$$
f_{\alpha,x} = \frac{(-1)^{|\alpha|} \partial^{\alpha}(T_x \delta)}{Cw_M(x)} \in (B')^\circ \subseteq B^\circ
$$
for all $x \in \R^d$ and $|\alpha| \leq n$.  We obtain that
\begin{align*}
&\sup_{\varphi \in B} \| \varphi\|_{\mathcal{B}^n_{1/w_M}} = C \sup_{\varphi \in B}\max_{|\alpha| \leq n} \sup_{x \in \R^d} |\langle f_{\alpha,x}, \varphi \rangle | \leq C. 
\qedhere
\end{align*}
\end{proof}
\begin{remark}
The implication $(ii) \Rightarrow (iii)$ in Theorem \ref{char-UB-smooth} also follows from De Wilde's open mapping theorem and the fact that the strong dual of an $(LF)$-space is strictly webbed \cite[Prop.\ IV.3.3]{DeWilde-cls}.
\end{remark}
\section{Weighted $\mathcal{D}'_{L^1}$ spaces}\label{sect-weighted-DL1}
 Let $w$ be a positive measurable function and set $\mathcal{D}'_{L^1_w}(\R^d) = (\dot{\mathcal{B}}_{1/w}(\R^d))'_b$. In this short section, we apply Theorem \ref{char-UB-smooth} to study the structural and topological properties of  $\mathcal{D}'_{L^1_w}(\R^d)$.
\begin{theorem}\label{weighted-dual-DL1}
Let $w$ be a positive measurable function and assume that
\begin{equation}
g = \operatorname*{ess \: sup}_{x \in \R^d}\frac{w(x + \: \cdot \:)}{w(x)} \in L^\infty_{\operatorname{loc}}(\R^d).
\label{gen-cond-ofnie}
\end{equation}
Then, the following properties hold:
\begin{itemize}
\item[$(i)$] $\mathcal{D}'_{L^1_w}(\R^d) = \mathcal{O}'_C(\mathcal{D}, L^1_w)$ topologically.
\item[$(ii)$]  The strong dual of $\mathcal{D}'_{L^1_w}(\R^d)$ is given by $\mathcal{B}_{1/w}(\R^d)$. More precisely, let $\psi \in \mathcal{D}(\R^d) \backslash \{0\}$ and $\gamma \in \mathcal{D}(\R^d)$ be a synthesis window for $\psi$. Then, the mapping $\mathcal{B}_{1/w}(\R^d) \rightarrow (\mathcal{D}'_{L^1_w}(\R^d))'_{b}$ given by
$$
\varphi \mapsto \left( f \mapsto \frac{1}{(\gamma,\psi)_{L^2}}\int \int_{\R^{2d}} V_\psi f(x,\xi) V_{\overline{\gamma}}\varphi(x,-\xi) \dx \dxi \right)
$$
is a topological isomorphism.
\item[$(iii)$] $\mathcal{D}'_{L^1_w}(\R^d)$ is ultrabornological.
\end{itemize}
\end{theorem}
\begin{proof}
We may assume that $w$ is continuous. For otherwise, consider the continuous weight $\widetilde w =  w\ast \varphi$, where $\varphi \in \mathcal{D}(\R^d)$ is non-negative and satisfies $\int_{\R^d} \varphi(t) \dt = 1$. Set $K = \operatorname{supp} \varphi$.  Then,
$$
\|g\|_{K}^{-1}w(x) \leq \widetilde{w}(x)  \leq \|\check{g}\|_{K}w(x), \qquad x \in \R^d.
$$
Hence, $\mathcal{B}_{1/w}(\R^d) =\mathcal{B}_{1/\widetilde{w}}(\R^d)$, $\dot{\mathcal{B}}_{1/w}(\R^d) =\dot{\mathcal{B}}_{1/\widetilde{w}}(\R^d)$ and $L^1_w(\R^d) =L^1_{\widetilde{w}}(\R^d)$. Moreover, 
$$
\widetilde{w}(x+t) = \int_{\R^d} w(x+t-y) \varphi(y) \dy \leq g(t) \int_{\R^d} w(x-y) \varphi(y) \dy = g(t) \widetilde{w}(x), \qquad x,t \in \R^d,
$$
whence
$$
\sup_{x \in \R^d}\frac{\widetilde{w}(x + \: \cdot \:)}{\widetilde{w}(x)} \in L^\infty_{\operatorname{loc}}(\R^d).
$$
From now on, we assume that $w$ is continuous. Set $\mathcal{W} = (w)_{N}$, a constant weight system, and notice that $\mathcal{W}^\circ$ satisfies $(\Omega)$. Properties $(i)$--$(iii)$  therefore follow immediately from Corollary \ref{bidual-Omega} and Theorem \ref{char-UB-smooth}.
\end{proof}
\begin{remark}
 When  $w = 1$, Theorem \ref{weighted-dual-DL1} is essentially due to  Schwartz \cite[pp.\ 201-203]{Schwartz} except for the topological identity  $
\mathcal{D}'_{L^1}(\R^d) = \mathcal{O}'_C(\mathcal{D}, L^1)$, which appears to be new even in this case; Schwartz only showed that these spaces coincide as sets 
and have the same bounded sets and null sequences \cite[p.\ 202]{Schwartz}. Structural and topological properties of weighted $\mathcal{D}'_{L^1}$ spaces have recently been studied in the broader context of distribution 
spaces associated to general translation-invariant Banach spaces \cite{D-P-V2015TIB}. In particular, the analogues of Schwartz' results were proved there, but, again, Theorem
\ref{weighted-dual-DL1}$(i)$ seems to be new in this setting.
\end{remark}

\begin{remark}\label{O-regular}
Observe that the function $g$ from \eqref{gen-cond-ofnie} is submultiplicative. Since every locally bounded submultiplicative function is exponentially bounded, the weight $w$ satisfies \eqref{gen-cond-ofnie} if and only if
$$
w(x+y) \leq C w(x)e^{\lambda|y|}, \qquad x, y \in \R^d,
$$
for some $C, \lambda > 0$. Moreover, in dimension one,  \eqref{gen-cond-ofnie} says that $w\circ \log$ is an $O$-regularly varying function \cite[Chap. 2]{B-G-T}. 

It is worth pointing out that  condition $\eqref{gen-cond-ofnie}$ holds if and only if $L_{w}^{1}(\R^{d})$ is translation-invariant and in this case the function $g$ turns out to be the weight function of its translation group (in the sense of \cite[Def.\ 1, p.\ 500]{D-P-V2015TIB}). In fact, the sufficiency of \eqref{gen-cond-ofnie} is easily established. Conversely, suppose that $L^1_{w}(\R^{d})$ is translation-invariant and consider the functions
$$
g (y)= \operatorname*{ess \: sup}_{x \in \R^d}\frac{w(x +y)}{w(x)} \qquad \mbox{and}\qquad \eta(y)=\|T_{y}\|_{L_{b}(L^1_{w})}, \qquad y \in \R^d.
$$
The function $g$ may of course take the value $\infty$, while $\eta\in L_{loc}^{\infty}(\mathbb{R}^d)$ because it is measurable, submultiplicative and everywhere finite (cf. \cite[Thm.\ 7.13.1, p.\ 253]{hille-p}). Clearly $\eta\leq g$. Fix $y\in \mathbb{R}^{d}$ and let $\alpha<g(y)$. We select a compact set $K \subset\{x\in\mathbb{R}^{d}\: : \: \alpha<w(x+y)/w(x)\}$ with positive Lebesgue measure and write $\chi_{K}$ for its characteristic function. Then,
$$
\alpha \int_{K}w(x)\mathrm{d}x\leq \|T_{y}\chi_{K}\|_{L^{1}_{w}}\leq \eta(y)\int_{K}w(x)\mathrm{d}x,
$$
whence, $\alpha\leq\eta(y)$. Since $\alpha$ and $y$ were arbitrary, we conclude that $g=\eta$ and the claim follows.
\end{remark}

\section{Convolutor spaces of Gelfand-Shilov spaces}\label{GS-convolutors}
As a second application of our general theory developed in Section \ref{L1-convolutors}, we discuss the convolutor spaces of general Gelfand-Shilov spaces, as introduced in \cite[Chap.\ II]{G-S}. 

Let $\mathcal{W} = (w_N)_N$ be an increasing weight system. We define\footnote{Gelfand and Shilov only considered the space $\mathcal{B}_{\mathcal{W}}(\R^d)$ for which they used the notation $\mathcal{K}\{M_p\}$, where $M_p = w_p$, $p \in \N$.} the ensuing Fr\'echet spaces
$$
\mathcal{B}_{\mathcal{W}}(\R^d) := \varprojlim_{N \in \N} \mathcal{B}_{w_N}(\R^d), \qquad  \dot{\mathcal{B}}_{\mathcal{W}}(\R^d) := \varprojlim_{N \in \N} \dot{\mathcal{B}}_{w_N}(\R^d).
$$
\begin{remark}\label{remark-P}
If $\mathcal{W} = (w_N)_N$ satisfies condition $(P)$ (cf.\ \cite[p.\ 87]{G-S}), i.e.,
$$
\forall N \in \N \, \exists  M > N: \lim_{|x| \to \infty} w_N(x)/w_M(x) = 0,
$$
then $\mathcal{B}_{\mathcal{W}}(\R^d) = \dot{\mathcal{B}}_{\mathcal{W}}(\R^d)$ and this space is an $(FS)$-space.
\end{remark}
Set  $\check{\mathcal{W}} = (\check{w}_N)_N$. For every $f \in (\dot{\mathcal{B}}_{\check{\mathcal{W}}}(\R^d))'$ the mapping
\begin{equation}
\dot{\mathcal{B}}_{\mathcal{W}}(\R^d) \rightarrow \mathcal{E}(\R^d), \quad \varphi \mapsto f \ast \varphi
\label{cont-conv}
\end{equation}
is well-defined and continuous (cf. Proposition \ref{convspaces-proposition-structural}). We define
$$
\mathcal{O}'_C(\dot{\mathcal{B}}_{\mathcal{W}}) := \{ f \in  (\dot{\mathcal{B}}_{\check{\mathcal{W}}}(\R^d))' \, : \, f \ast \varphi \in \dot{\mathcal{B}}_{\mathcal{W}}(\R^d) \mbox{ for all $\varphi \in \dot{\mathcal{B}}_{\mathcal{W}}(\R^d)$} \}.
$$
The closed graph theorem and the continuity of the mapping \eqref{cont-conv} imply that, for $f \in \mathcal{O}'_C(\dot{\mathcal{B}}_{\mathcal{W}})$ fixed, the mapping $\dot{\mathcal{B}}_{\mathcal{W}}(\R^d) \rightarrow \dot{\mathcal{B}}_{\mathcal{W}}(\R^d), \, \varphi \mapsto f \ast \varphi$ is continuous. We endow $\mathcal{O}'_C(\dot{\mathcal{B}}_{\mathcal{W}})$ with the initial topology with respect to the mapping
$$
\mathcal{O}'_C(\dot{\mathcal{B}}_{\mathcal{W}}) \rightarrow L_{b}(\dot{\mathcal{B}}_{\mathcal{W}}(\R^d),\dot{\mathcal{B}}_{\mathcal{W}}(\R^d)), \quad f \mapsto (\varphi \mapsto f \ast \varphi). 
$$
The goal of this section is to study the structural and topological properties of $\mathcal{O}'_C(\dot{\mathcal{B}}_{\mathcal{W}})$. We shall need to impose the following conditions on $\mathcal{W}$:
\begin{equation}
\forall N \in \N \, \exists M > N : w_N /w_M \in L^1(\R^d)
\label{L1-cond}
\end{equation}
and
\begin{equation}
\forall N \in \N \, \exists M,K \geq N \, \exists C > 0 \, \forall x,y \in \R^d : w_N(x+y) \leq C w_M(x) w_K(y).
\label{trans-inv}
\end{equation}
Clearly, \eqref{trans-inv} implies  \eqref{locally-bounded-increasing}. Clearly, \eqref{trans-inv} implies  \eqref{locally-bounded-increasing}. Moreover, one readily shows that \eqref{L1-cond} and \eqref{locally-bounded-increasing} imply $(P)$ (cf.\ Remark \ref{remark-P}) and thus $\mathcal{B}_{\mathcal{W}}(\R^d) = \dot{\mathcal{B}}_{\mathcal{W}}(\R^d)$. In such a case, we shall simply write $\mathcal{O}'_C(\dot{\mathcal{B}}_{\mathcal{W}}) = \mathcal{O}'_C(\mathcal{B}_{\mathcal{W}})$. It is worth mentioning that Gelfand and Shilov used condition \eqref{L1-cond} for deriving structural theorems for $(\mathcal{B}_{\mathcal{W}}(\R^d))'$ \cite[p.\ 113]{G-S}.

\begin{proposition}\label{eq-L1}
Let $\mathcal{W} =(w_N)_N$ be an increasing weight system satisfying \eqref{L1-cond} and \eqref{trans-inv}. Then, $\mathcal{O}'_C(\mathcal{B}_{\mathcal{W}}) = \mathcal{O}'_C(\mathcal{D}, L^1_{\mathcal{W}})$ topologically.
\end{proposition}

We need  some results in preparation for the proof of Proposition \ref{eq-L1}.  Define the following Fr\'echet space of continuous functions

$$
\mathcal{W}_{\operatorname{pol},0}C(\R^{2d}) := \varprojlim_{N \in \N} C(w_N\otimes (1+|\:\cdot\:|)^{N})_0(\R^{2d}).
$$
Lemmas \ref{STFT-test-smooth} and \ref{double-int-test-smooth} together with \eqref{reconstruction-D-dual} immediately imply the next result.

\begin{proposition}\label{STFT-test-char-smooth-proj-1}
Let $\psi \in \mathcal{D}(\R^d)$ and let $\mathcal{W} = (w_N)_N$ be an increasing weight system satisfying \eqref{locally-bounded-increasing}. Then, the mappings 
$$
V_\psi: \dot{\mathcal{B}}_{\mathcal{W}}(\R^d) \rightarrow \mathcal{W}_{\operatorname{pol},0}C(\R_{x,\xi}^{2d}) 
$$ 
and 
$$
V^\ast_\psi:  \mathcal{W}_{\operatorname{pol},0}C(\R_{x,\xi}^{2d})  \rightarrow \dot{\mathcal{B}}_{\mathcal{W}}(\R^d)
$$
are well-defined and continuous. Moreover, if $\psi \neq 0$ and $\gamma \in \mathcal{D}(\R^d)$ is a synthesis window for $\psi$, then
$$
\frac{1}{(\gamma, \psi)_{L^2}} V^\ast_\gamma \circ V_\psi = \operatorname{id}_{\dot{\mathcal{B}}_{\mathcal{W}}(\R^d)}.
$$
\end{proposition}
We obtain the following corollary; it is the weighted analogue of a classical result of Schwartz \cite[p.\ 200]{Schwartz}.
\begin{corollary}\label{konnieanders}
Let $\mathcal{W}= (w_N)_N$ be an increasing weight system satisfying \eqref{locally-bounded-increasing}. Then, $\mathcal{D}_{L^1_{\mathcal{W}}}(\R^d) \subseteq \dot{\mathcal{B}}_{\mathcal{W}}(\R^d)$ with continuous inclusion mapping. If, in addition,  $\mathcal{W}$ satisfies \eqref{L1-cond}, then $\mathcal{D}_{L^1_{\mathcal{W}}}(\R^d) = \dot{\mathcal{B}}_{\mathcal{W}}(\R^d) = \mathcal{B}_{\mathcal{W}}(\R^d)$.
\end{corollary}
\begin{proof}
We start by showing that $\mathcal{D}_{L^1_{\mathcal{W}}}(\R^d)$ is continuously included in $\dot{\mathcal{B}}_{\mathcal{W}}(\R^d)$. By Proposition \ref{STFT-test-char-smooth-proj-1}, it suffices to show that
$
V_\psi: \mathcal{D}_{L^1_\mathcal{W}}(\R^d) \rightarrow \mathcal{W}_{\operatorname{pol},0}C(\R^{2d}) 
$ 
is well-defined and continuous. This can be done by modifying the proof of Lemma \ref{STFT-test-smooth}. Next, observe that condition \eqref{L1-cond} implies that $\mathcal{B}_{\mathcal{W}}(\R^d)$ is continuously included in $ \mathcal{D}_{L^1_\mathcal{W}}(\R^d)$, whence the second part follows from the first one.
\end{proof}
\begin{proof}[Proof of Proposition \ref{eq-L1}]
The continuous inclusion $\mathcal{O}'_C({\mathcal{B}}_{\mathcal{W}}) \subseteq \mathcal{O}'_C(\mathcal{D}, L^1_{\mathcal{W}})$ is clear from \eqref{L1-cond}. For the converse inclusion, we introduce the ensuing space
$$
\mathcal{O}'_C({\mathcal{B}}_{\mathcal{W}}, L^1_{\mathcal{W}}) := \{ f \in  ({\mathcal{B}}_{\check{\mathcal{W}}}(\R^d))' \, : \, f \ast \varphi \in L^1_{\mathcal{W}}(\R^d) \mbox{ for all $\varphi \in {\mathcal{B}}_{\mathcal{W}}(\R^d)$} \}
$$
endowed with the initial topology with respect to the mapping 
$$
\mathcal{O}'_C({\mathcal{B}}_{\mathcal{W}},L^1_{\mathcal{W}}) \rightarrow L_b({\mathcal{B}}_{\mathcal{W}}(\R^d),L^1_{\mathcal{W}}(\R^d)), \quad f \mapsto (\varphi \mapsto f \ast \varphi).
$$
Similarly as in Corollary \ref{conv-b-dot} and the first part of the proof of Theorem \ref{thm-predual}, but by using \eqref{trans-inv} instead of \eqref{locally-bounded-increasing}, one can show that $({\mathcal{B}}_{\mathcal{W}^\circ}(\R^d))'_{bs}$ is continuously included in $\mathcal{O}'_C({\mathcal{B}}_{\mathcal{W}}, L^1_{\mathcal{W}})$. Next, we notice that an element $f \in  ({\mathcal{B}}_{\check{\mathcal{W}}}(\R^d))'$ belongs to $\mathcal{O}'_C({\mathcal{B}}_{\mathcal{W}},L^1_{\mathcal{W}})$ if and only if $f \ast \varphi \in \mathcal{D}_{L^1_{\mathcal{W}}}(\R^d)$ for all $\varphi \in {\mathcal{B}}_{\mathcal{W}}(\R^d)$ and that the topology of $\mathcal{O}'_C({\mathcal{B}}_{\mathcal{W}},L^1_{\mathcal{W}})$ coincides with the initial topology with respect to the mapping 
$$
\mathcal{O}'_C({\mathcal{B}}_{\mathcal{W}},L^1_{\mathcal{W}}) \rightarrow L_b({\mathcal{B}}_{\mathcal{W}}(\R^d),\mathcal{D}_{L^1_{\mathcal{W}}}(\R^d)), \quad f \mapsto (\varphi \mapsto f \ast \varphi).
$$
Hence, the result follows from Corollary \ref{konnieanders}.
\end{proof}
Proposition \ref{eq-L1} and Theorem \ref{char-UB-smooth}  imply the following  result. 
\begin{theorem}\label{char-UB-smooth-GS}
Let $\mathcal{W} = (w_N)_{N}$ be an increasing weight system satisfying \eqref{L1-cond} and \eqref{trans-inv}. Then, the following statements are equivalent:
\begin{itemize}
\item[$(i)$] $\mathcal{W}^\circ$ satisfies $(\Omega)$.
\item[$(ii)$] $\mathcal{O}'_C({\mathcal{B}}_{\mathcal{W}})$ is ultrabornological.
\item[$(iii)$] $({\mathcal{B}}_{\mathcal{W}^\circ}(\R^d))'_{b}= \mathcal{O}'_C({\mathcal{B}}_{\mathcal{W}})$.
\end{itemize}
\end{theorem}
Finally, we apply Theorem \ref{char-UB-smooth-GS} to study the convolutor spaces of several Gelfand-Shilov spaces that are frequently used in the literature. To this end, we evaluate conditions \eqref{L1-cond}, \eqref{trans-inv} and $(\Omega)$ for two classes of increasing weight systems. Namely, let $w$ be a positive continuous increasing function on $[0,\infty)$ and extend $w$ to $\R^d$ by setting $w(x) = w(|x|)$ for $x \in \R^d$. We define the following increasing weight systems:
$\mathcal{W}_{w} = (e^{Nw})_N$ and $\widetilde{\mathcal{W}}_{w} = (e^{w(N\:\cdot\:)})_N$.
The function $w$ is said to satisfy $(\alpha)$ (cf.\ \cite{B-M-T}) if there are $C'_0, H' \geq 1$ such that 
$$
w(2t) \leq H'w(t) + \log C'_0, \qquad t \geq 0. 
$$
The proof of the next result is simple and therefore omitted.
\begin{proposition}\label{char-omega-1}
Let $w$ be a positive continuous increasing function on $[0,\infty)$ with $\log(1 + t) = O(w(t))$ satisfying $(\alpha)$. Then, $\mathcal{W}_{w}$ satisfies \eqref{L1-cond} and  \eqref{trans-inv}, while $\mathcal{W}^\circ_{w}$ satisfies $(\Omega)$. Consequently, $\mathcal{O}'_C({\mathcal{B}}_{\mathcal{W}_{w}})$ is ultrabornological.
\end{proposition}

\begin{example}\label{ex 1 convolutors smooth}For $w(t) = \log(e+t)$ we have that $\mathcal{B}_{\mathcal{W}_w}(\R^d) = \mathcal{S}(\R^d)$. Proposition \ref{char-omega-1} implies that the space of convolutors $\mathcal{O}'_C(\mathcal{S}) = \mathcal{O}'_C(\R^d)$ is ultrabornological. This  was first proved by Grothendieck \cite[Chap.\ II, Thm.\ 16, p.\ 131]{Grothendieck}. He showed  that  $\mathcal{O}'_C(\R^d)$ is isomorphic to a complemented subspace of $s \widehat{\otimes} s'$ and proved that the latter space is ultrabornological. 
\end{example}

\begin{example}\label{exp-example}
 For $w(t) = t$ we have that $\mathcal{B}_{\mathcal{W}_w}(\R^d) = \mathcal{K}_1(\R^d)$, the space of exponentially decreasing smooth functions \cite{Hasumi}. Proposition \ref{char-omega-1} yields that the space of convolutors $\mathcal{O}'_C(\mathcal{K}_1)$ is ultrabornological. This  was  first claimed by Ziele\'{z}ny in \cite{Zielezny} but his argument seems to contain major gaps (in particular, the proof of \cite[Thm.\ 9]{Zielezny} does not seem to be correct). 
\end{example}
\begin{example} 
\label{ex 3 convolutors smooth}
More generally, let $w(t) = t^p$, $p > 0$, and set  $\mathcal{B}_{\mathcal{W}_w}(\R^d) = \mathcal{K}_p(\R^d)$. For $p > 1$ the convolutor spaces $\mathcal{O}'_C(\mathcal{K}_p)$ were studied in \cite{S-Z}, where  the hypoelliptic convolution operators in $(\mathcal{K}_p(\R^d))'$ are characterized in terms of their Fourier transform.
However, the topological properties of $\mathcal{O}'_C(\mathcal{K}_p)$ do not seem to have been studied yet. Proposition \ref{char-omega-1} implies that the space of convolutors $\mathcal{O}'_C(\mathcal{K}_p)$ is ultrabornological for each $p > 0$.
 \end{example}
Next, we treat increasing weight systems of type $\widetilde{\mathcal{W}}_{w}$.

\begin{proposition} \label{char-omega-2}
Let $w$ be a positive continuous increasing function on $[0,\infty)$ with $\log(1 + t) = O(w(t))$ such that
\begin{equation}
2w(t) \leq w(Ht) + \log C_0, \qquad t \geq 0, 
\label{M2like}
\end{equation}
for some $H, C_0 \geq 1$. Then, $\widetilde{\mathcal{W}}_{w}$ satisfies \eqref{L1-cond} and  \eqref{trans-inv}, while $\widetilde{\mathcal{W}}^\circ_{w}$ satisfies $(\Omega)$ if and only if $w$ satisfies $(\alpha)$. Consequently, $\mathcal{O}'_C({\mathcal{B}}_{\widetilde{\mathcal{W}}_{w}})$ is ultrabornological if and only if $w$ satisfies $(\alpha)$.
\end{proposition}
\begin{proof}  By iterating \eqref{M2like}, we obtain that for every $\lambda > 0$ there are $L, B \geq 1$ 
such that $\lambda w(t) \leq w(Lt) + B$ for all  $t \geq 0$. Condition $\eqref{L1-cond}$ therefore follows from the assumption $\log(1 + t) = O(w(t))$. Next,  \eqref{trans-inv} is a consequence of the fact that $w$ is increasing. We  now show that $\widetilde{\mathcal{W}}^\circ_{w}$ satisfies $(\Omega)$ if and only if $w$ satisfies $(\alpha)$. For this, observe that $\widetilde{\mathcal{W}}^\circ_{w}$ satisfies $(\Omega)$ if and only if there is $M \geq 1$ such that for all $K \geq M$ there are $C, S \geq 1$ such that
\begin{equation}
w(Kt) - w(t) \leq C(w(Mt) - w(t)) + S, \qquad t \geq 0. 
\label{Omega-equiv-2}
\end{equation}
This shows that  $w$ satisfies $(\alpha)$ if  $\widetilde{\mathcal{W}}^\circ_{w}$ satisfies $(\Omega)$. Conversely, suppose that $w$ satisfies $(\alpha)$. We shall show \eqref{Omega-equiv-2} for $M = H$, where $H$ is the constant occurring in \eqref{M2like}. Let $K \geq H$ be arbitrary. By iterating $(\alpha)$, we obtain that there are $L,B \geq 1$ such that
$w(Kt) \leq Lw(t) +  B$ for all $t \geq 0$.
Hence,
$$
w(Kt) - w(t) \leq (L-1) w(t) + B  \leq (L-1)(w(Ht) - w(t)) + (L-1)\log C_0 + B. \qedhere
$$
\end{proof}
\begin{remark}\label{remark-fast}
Since condition $(\alpha)$ implies that $w$ is polynomially bounded, Proposition \ref{char-omega-2} in particular yields that $\mathcal{O}'_C({\mathcal{B}}_{\widetilde{\mathcal{W}}_{w}})$ is not ultrabornological for any function $w$ satisfying \eqref{M2like} that is not polynomially bounded. Concrete examples are given by: $w(t) = \exp(t^\sigma \log^{\tau}(1+ t))$ with $\sigma > 0$ and $\tau \geq 0$ or $\sigma = 0$ and $\tau > 1$. On the other hand, there are also polynomially bounded functions $w$ for which $\mathcal{O}'_C({\mathcal{B}}_{\widetilde{\mathcal{W}}_{w}})$ fails to be ultrabornological, as shown by the following example.
\end{remark}

\begin{example}\label{remark-slow}
\emph{For any $\sigma > 0$ there is a  function $w$ with $\log(1 + t) = O(w(t))$ and $w(t) = O(t^\sigma)$ such that $w$ satisfies \eqref{M2like} but violates $(\alpha)$}. We shall use some results about weight sequences and their associated function; we refer to \cite{Komatsu, B-M-M} for more information on this topic. In \cite[Example 3.3]{Langenbruch-example}, Langenbruch constructed a weight sequence $(M_p)_{p \in \N}$ satisfying $(M.1)$, $(M.2)$ and $(M.3)'$ but not 
$$
\sup_{p \in \N} \frac{m_p}{m_{Qp}} < \infty, \qquad \mbox{for some $Q \in \Z_+$},
$$
where $m_p = M_{p+1}/M_{p}$, $p \in \N$. Hence, the associated function $M$ of $(M_p)_{p \in \N}$ is a weight function such that $\log(1 + t) = o(M(t))$ \cite[p.\ 49]{Komatsu}, $M(t) = o(t)$ \cite[Lemma 4.1]{Komatsu} and $M$ satisfies \eqref{M2like} \cite[Prop.\ 3.6]{Komatsu} but violates $(\alpha)$ \cite[Prop.\ 13]{B-M-M}. Therefore, $w(t) = M(t^\sigma)$ satisfies all requirements. 
\end{example}

We now further specialize Proposition \ref{char-omega-2} to a class of weights introduced by Gelfand and Shilov in \cite[Chap.\ IV, Appendix 2]{G-S}. Let $\mu$ be an increasing positive function on $[0,\infty)$ and consider $w(t) = \int_0^{t} \mu(s) \dss$, $t \geq 0$. We write $\mathcal{K}_w(\R^d) = \mathcal{B}_{\widetilde{\mathcal{W}}_w}(\R^d)$. 
\begin{theorem}\label{thm-UB-GS-2}
 Let $\mu$ be an increasing positive function on $[0,\infty)$ and set $w(t) = \int_0^{t} \mu(s) \dss$. Then, the following statements are equivalent:
 \begin{itemize}
 \item[$(i)$] $\mathcal{O}'_C(\mathcal{K}_w)$ is ultrabornological.
 \item[$(ii)$] $w$ satisfies $(\alpha)$.
 \item[$(iii)$] There are $L, B \geq 1$ such that
\begin{equation}
 \mu(2s) \leq L \mu(s) + \frac{B}{s}, \qquad s > 0.
 \label{mu-cond}
 \end{equation}
 \end{itemize} 
\end{theorem}
\begin{proof}
Clearly, $\log(1 + t) = O(w(t))$ (in fact,  $t = O(w(t))$). Next, we verify that  $w$ satisfies \eqref{M2like}. Since $\mu$ is increasing, we obtain that
$$
2w(t) \leq \int_0^t \mu(s) \dss +  \int_0^t \mu(s+t) \dss = \int_0^t \mu(s) \dss +  \int_t^{2t} \mu(s) \dss =  \int_0^{2t} \mu(s) \dss = w(2t).
$$
Hence, in view of Proposition \ref{char-omega-2}, it suffices to show that $w$ satisfies $(\alpha)$ if and only if $\mu$ satisfies \eqref{mu-cond}. Suppose first that $\mu$ satisfies \eqref{mu-cond}. Then,
\begin{align*}
&w(2t) - w(t) = \int^{2t}_t \mu(s) \dss =  2\int^t_{t/2} \mu(2s) \dss \leq 2\int^t_{t/2} \left(L \mu(s) + \frac{B}{s}\right)\dss \\
&\leq 2L \int^t_0 \mu(s)\dss + 2(\log2)B = 2Lw(t) + 2(\log2)B.
\end{align*}
Conversely, assume that $w$ satisfies $(\alpha)$. Applying $(\alpha)$ twice, we obtain that
$$
w(4s) -w(2s) \leq (H'-1)w(2s) + \log C'_0 \leq H'(H'-1)w(s) + H' \log C'_0.
$$
Since $\mu$ is increasing, we have that
\begin{align*}
&\mu(2s) = \frac{\mu(2s)}{\log2} \int^{4s}_{2s} \frac{1}{u} \du \leq\frac{1}{\log2} \int^{4s}_{2s} \frac{\mu(u)}{u} \du \leq \frac{1}{2(\log2)s} \int^{4s}_{2s} \mu(u) \du \\
&= \frac{1}{2(\log2)s}(w(4s) - w(2s)) \leq \frac{H'(H'-1)}{2(\log2)s}w(s) + \frac{H'\log C'_0}{2(\log2)s} \\
&= \frac{H'(H'-1)}{2(\log2)s}\int^s_0\mu(u) \du + \frac{H'\log C'_0}{2(\log2)s} \leq \frac{H'(H'-1)\mu(s)}{2(\log2)} + \frac{H'\log C'_0}{2(\log2)s}.
\qedhere
\end{align*}
\end{proof}

\begin{remark}\label{remark-KM}
The topological properties of the convolutor spaces $\mathcal{O}'_C(\mathcal{K}_w)$ were discussed by Abdullah. On \cite[p.\ 179]{Abdullah-2} he states that the spaces $\mathcal{O}'_C(\mathcal{K}_w)$ are always ultrabornological, but he does not provide a proof of this assertion. Theorem \ref{thm-UB-GS-2} shows that his claim is actually false. By Remark \ref{remark-fast}, the space $\mathcal{O}'_C(\mathcal{K}_w)$ is not ultrabornological if $w$ is not polynomially bounded. Furthermore, one can use Example \ref{remark-slow} to construct weights $w$ of the form $w(t) = \int_0^{t} \mu(s) \dss$ with $w(t) = O(t^{1+\sigma})$, for a fixed but arbitrary $\sigma > 0$, such that $w$ does not satisfy $(\alpha)$ (take $\mu$ to be any function satisfying the conditions from Example \ref{remark-slow} and use Theorem  \ref{thm-UB-GS-2}$(iii)$).
\end{remark}

\end{document}